\documentclass[a4paper,11pt]{article}
\usepackage[affil-it]{authblk}
\usepackage{amsmath}
\usepackage{amsthm}
\usepackage{caption}
\usepackage{amssymb}
\usepackage{bbm, dsfont}

\usepackage{hyperref}
\numberwithin{equation}{section}
\newtheorem{theorem}{Theorem}[section]
\newtheorem{definition}{Definition}[section]

\newtheorem{lemma}{Lemma}[section]

\newtheorem{corollary}{Corollary}[section]

\newtheorem{problem}{Problem}[section]

\newtheorem{proposition}{Proposition}[section]

\newtheorem{remark}{Remark}[section]

\usepackage{amsfonts}
\usepackage{array}
\usepackage{verbatim}
\usepackage{enumerate}

\usepackage{twoopt}

\usepackage{color}

\newcommand{\norm}[1]{\left\|{#1}\right\|}

\newcommand{\ip}[2]{\left<#1,#2\right>}  

\newcommand{\p}[2]{\frac{\partial{#1}}{\partial{#2}}} 

\newcommand{\tos}{\xrightarrow{s}}  
\newcommand{\tow}{\rightharpoonup}
\newcommand{\tosr}{\xrightarrow{s.r.}}
\newcommand{\tonr}{\xrightarrow{n.r.}}

\newcommand{\dom}[1]{\spac{D}({#1})} 
\newcommand{\bounded}[1]{\spac{B}\left({#1}\right)}  

\newcommand{\+}[1]{\ensuremath{\boldsymbol{#1}}} 
\newcommand{\?}[1]{\ensuremath{\mathbf{#1}}}

\DeclareMathOperator{\negeig}{neg} 

\newcommand{\<}[1]{\left<{#1}\right>} 

\newcommand{\oper}[1]{\mathcal{#1}} 
\newcommand{\spac}[1]{\mathfrak{#1}}
\newcommand{\sform}[1]{\mathfrak{\lowercase{#1}}}

\newcommand{\M}[1][\lambda]{\ensuremath{\+{\oper{M}}^{#1}}} 
\newcommand{\Mt}[1][\lambda]{\ensuremath{\+{\widetilde{\oper{M}}}^{#1}}} 
\newcommandtwoopt{\Me}[2][\lambda][\varepsilon]{\ensuremath{\+{\oper{M}}^{#1}_{#2}}} 

\newcommand{\D}{\oper{D}_\pm}
\newcommand{\Dt}{\widetilde{\oper{D}}_\pm}
\newcommand{\J}[1][\lambda]{\+{\oper{J}}^{#1}}
\newcommand{\Jt}[1][\lambda]{\widetilde{\+{\oper{J}}}^{#1}}
\newcommand{\Q}[1][\lambda]{\oper{Q}_\pm^{#1}}
\newcommand{\Qt}[1][\lambda]{\widetilde{\oper{Q}}_{\pm}^{#1}}
\newcommand{\Qtsym}[1][\lambda]{\widetilde{\oper{Q}}_{\pm,sym}^{#1}}
\newcommand{\Qtskew}[1][\lambda]{\widetilde{\oper{Q}}_{\pm,skew}^{#1}}

\newcommand{\Ls}{\spac{L}}
\newcommand{\Ns}{\spac{N}}

\title{Instabilities of the relativistic Vlasov-Maxwell system on unbounded domains}
\author{Jonathan Ben-Artzi \thanks{Email: \texttt{J.Ben-Artzi@imperial.ac.uk}}} 
\affil{Department of Mathematics\\Imperial College London}
\author{Thomas Holding \thanks{Email: \texttt{T.J.Holding@maths.cam.ac.uk}}}
\affil{Cambridge Centre for Analysis\\University of Cambridge}
\begin{document}
\maketitle

\abstract{The relativistic Vlasov-Maxwell system describes the evolution of a collisionless plasma. The problem of linear instability of this system is considered in two physical settings: the so-called ``one and one-half'' dimensional case, and the three dimensional case with cylindrical symmetry. Sufficient conditions for instability are obtained in terms of the spectral properties of certain Schr\"odinger operators that act on the spatial variable alone (and not in full phase space). An important aspect of these conditions is that they do not require any boundedness assumptions on the domains, nor do they require monotonicity of the equilibrium.}

\tableofcontents

\section{Introduction}
We obtain new linear instability results for plasmas governed by the relativistic Vlasov-Maxwell (RVM) system of equations. The main unknowns are two functions $f^\pm=f^\pm(t,x,v)\geq0$ measuring the density of positively and negatively charged particles that at time $t\in[0,\infty)$ are located at the point $x\in \mathbb R^d$ and have momentum $v\in\mathbb R^d$. The densities $f^\pm$ evolve according to the Vlasov equations
	\begin{equation}\label{eq:vlasov}
	\frac{\partial f}{\partial t}^\pm+\+{\hat{v}}\cdot\nabla_{\+x} f^\pm+\mathbf F^\pm\cdot\nabla_{\+v}f^\pm=0
	\end{equation}
where $\+{\hat{v}}=\+v/\sqrt{1+|\+v|^2}$ is the relativistic velocity (the speed of light is taken to be $1$ for simplicity) and where $\mathbf F^\pm=\mathbf F^\pm(t,\+x,\+v)$ is the Lorentz force, given by
	\begin{equation*}
	\mathbf F^\pm=\pm\left(\mathbf E+\mathbf E^{ext}+\+{\hat{v}}\+\times (\mathbf B+\mathbf B^{ext})\right)
	\end{equation*}
with $\mathbf E=\mathbf E(t,\+x)$ and $\mathbf B=\mathbf B(t,\+x)$ being the electric and magnetic fields, respectively, and $\mathbf E^{ext},\mathbf B^{ext}$ possible fixed external fields. The \emph{self-consistent} fields are retrieved from Maxwell's equations
	\begin{equation*} 
	\nabla\cdot{\mathbf E}=\rho,\quad
	\nabla\cdot{\mathbf B}=0,\quad
	\nabla\times{\mathbf E}=-\frac{\partial{\mathbf B}}{\partial t},\quad
	\nabla\times{\mathbf B}={\mathbf j}+\frac{\partial {\mathbf E}}{\partial t},
	\end{equation*}
where
	\begin{equation}\label{eq:rho}
	\rho=\rho(t,\+x)=\int (f^+-f^-)\;d\+v
	\end{equation}
is the charge density and
	\begin{equation}\label{eq:j}
	{\mathbf j}={\mathbf j}(t,\+x)=\int \+{\hat{v}} (f^+-f^-)\;d\+v
	\end{equation} 
is the current density. In addition to the speed of light, we have taken all other constants that typically appear in these equations (such as the particle masses) to be $1$ so as to keep the notation simple.

\paragraph{Novelty of the results.}
Let us mention the main novel aspects of our instability results:

\emph{Unbounded domains:} our problems are set in unbounded domains (as opposed to domains with boundaries or periodic domains).  One consequence is that the spectrum of the Laplacian (which shall appear prominently) has an essential part.

\emph{Non-monotone equilibrium:}  we do not assume that the equilibrium in question is (strongly) monotone (see \eqref{eq:monotone} below). Many estimates in previous works rely heavily on monotonicity assumptions.

\emph{Construction of equilibria:}  in \autoref{sec:examples} we construct nontrivial equilibria in the \emph{unbounded, compactly supported} $1.5d$ case. Previously, this has been done in the periodic setting by means of a perturbation argument about the trivial solution which is a center (in the dynamical systems sense). The proof here relies on a fixed point argument.

\subsection{Main results}
For the convenience of the reader, we provide the full statements of our results here, although some necessary definitions are too cumbersome. We shall refer to the later sections for these definitions.

\paragraph{The physical setting.}
As is explained in detail below we consider two problems: the \emph{$1.5$ dimensional} case and the \emph{$3$ dimensional} case with \emph{cylindrical symmetry}. We shall refer to these two settings as the \emph{$1.5d$ case} and the \emph{$3d$ case}, respectively, for brevity. In a nutshell, we consider these settings because they provide enough structure so that basic existence and uniqueness results hold and because they possess well-known conserved quantities which may be written explicitly.

\paragraph{The equilibrium.}
The conserved quantities mentioned above -- the microscopic energy $e^\pm$ and momentum $p^\pm$ -- are the subject of further discussion below (see \eqref{eq:constants-1.5} for the $1.5d$ case and \eqref{eq:constants-3} for the $3d$ case), however we stress the fact that they are functions that satisfy the \emph{time-independent} Vlasov equations. Hence any  functions of the form
	\begin{equation}\label{eq:jeans}
	f^{0,\pm}(\+x,\+v)=\mu^\pm(e^\pm,p^\pm)
	\end{equation}
are equilibria of the corresponding Vlasov equations (for positively or negatively charged particles). The converse statement -- that any equilibrium may be written in this form is called Jeans' theorem \cite{Jeans1915}. In \autoref{sec:examples} we show how to obtain an actual equilibrium of the full Vlasov-Maxwell system in both physical settings. When there is no room for confusion we simply write $\mu^\pm(e,p)$ or $\mu^\pm$ instead of $\mu^\pm(e^\pm,p^\pm)$. We shall always assume that
	\begin{equation*} 
	0\leq f^{0,\pm}(\+x,\+v)\in C^1\text{ have compact support $\Omega$ in the $\+x$-variable}.
	\end{equation*}
Again, the existence of such equilibria is the subject of \autoref{sec:examples}. In addition, we must assume that
	\begin{equation}\label{eq:weight}
	\begin{split}
	&\text{there exist weight functions }w^\pm=c(1+|e^\pm|)^{-\alpha}\\
	&\text{where }\alpha>\text{dimension of momentum space, and }c>0,
	\end{split}
	\end{equation}
such that the integrability condition
	\begin{equation}\label{eq:int-condition}
	\left(\left|\frac{\partial\mu^\pm}{\partial e}\right|+\left|\frac{\partial\mu^\pm}{\partial p}\right|\right)(e^\pm,p^\pm)<w^\pm(e^\pm)
	\end{equation}
holds. This implies that $\int\left(|\mu^\pm_e|+|\mu^\pm_p|\right)d\+xd\+v<\infty$ in both the $1.5d$ and $3d$ cases, where we have abbreviated the writing of the partial derivatives of $\mu^\pm$. This abbreviated notation shall be used throughout the paper. It is typically \cite{Lin2007} assumed that
	\begin{equation}\label{eq:monotone}
	\mu^\pm_e<0\quad\text{whenever }\mu^\pm>0.
	\end{equation}
We call this a \emph{strong monotonicity condition.} \textbf{We do \underline{not} make any such assumption.}

\paragraph{The main results.}\label{Intro:main results}
To facilitate the understanding of our main results we state them now, trying not obscure the big picture with technical details. Hence we attempt to only extract those aspects of the statements that are crucial for understanding, while refering to later sections for some additional definitions. First we define our precise notion of instability:
\begin{definition}[Spectral instability]
We say that a given equilibrium $\mu^\pm$ is \emph{spectrally unstable}, if the system linearised around it has a purely growing mode solution of the form
	\begin{equation}\label{eq:purely-growing}
	\left(e^{\lambda t}f^\pm(\+x,\+v),e^{\lambda t}\mathbf{E}(\+x),e^{\lambda t}\mathbf{B}(\+x)\right),
	\quad
	\lambda>0.
	\end{equation}
\end{definition}
We also need the following definition:
\begin{definition}
Given a self-adjoint operator (bounded or unbounded) $\oper{A}$, we denote by $\negeig(\oper{A})$ the number of negative eigenvalues (counting multiplicity) that it has whenever there is a finite number of such eigenvalues.
\end{definition}

In our first result we obtain a sufficient condition for spectral instability of equilibria in the $1.5d$ case. The condition is expressed in terms of spectral properties of certain operators that act on functions of the spatial variable alone.
\begin{theorem}[Spectral instability: $1.5d$ case]\label{theorem1}
Let $f^{0,\pm}(x,\+v)=\mu^\pm(e,p)$ be an equilibrium of the $1.5d$ system \eqref{eq:rvm-1.5} satisfying \eqref{eq:int-condition}. There exist self-adjoint Schr\"odinger operators $\mathcal A^0_1$ and $\mathcal A^0_2$ and a bounded  operator $\mathcal B^0$ (all defined in \eqref{the 0 operators 1.5d}) acting only on functions of the spatial variable (and not the momentum variable) such that the equilibrium is spectrally unstable if $0$ is not in the spectrum of $\mathcal A_1^0$ and
	\begin{equation}\label{eq:thm1-condition}
	\operatorname{neg}\left(\mathcal A_2^0+\left(\mathcal B^0\right)^*\left(\mathcal A_1^0\right)^{-1}\mathcal B^0\right)> \operatorname{neg}\left(\mathcal A_1^0\right).
	\end{equation}
\end{theorem}

The second result provides a similar statement in the $3d$ case with cylindrical symmetry, as discussed in further detail in \autoref{sec:intro-cyl} below.

\begin{theorem}[Spectral instability: $3d$ case]\label{theorem2}
Let $f^{0,\pm}(\+x,\+v)=\mu^\pm(e,p)$ be a cylindrically symmetric equilibrium of the RVM system satisfying \eqref{eq:int-condition}. There exist self-adjoint operators $\widetilde{\oper{A}}^0_1$,$\widetilde{\oper{A}}^0_2$ and $\widetilde{\oper{A}}^0_3$ and a bounded operator $\widetilde{\oper{B}}_1^0$ (all defined in \eqref{eq:the 0 operators 3d}) acting in the spatial variable alone (and not the momentum variable) such that the equilibrium is spectrally unstable if $0$ is not an $L^6$-eigenvalue of $\widetilde{\oper{A}}^0_3$ (see \autoref{def:L6kernel} below), $0$ is not an eigenvalue of $\widetilde{\oper{A}}^0_1$ ($0$ will always lie in the essential spectrum of $\widetilde{\oper{A}}^0_1$, but this is not the same as $0$ being an eigenvalue) and
	\begin{equation}\label{eq:thm2-condition}
	\operatorname{neg}\left(\widetilde{\oper{A}}_2^0+\left(\widetilde{\oper{B}}_1^0\right)^*\left(\widetilde{\oper{A}}_1^0\right)^{-1}\widetilde{\oper{B}}_1^0\right)
	> \operatorname{neg}\left(\widetilde{\oper{A}}_1^0\right)+\operatorname{neg}\left(\widetilde{\oper{A}}_3^0\right).
	\end{equation}
\end{theorem}
Let us make precise the notion of an $L^6$-eigenvalue.

\begin{definition}[$L^6$-eigenvalue]\label{def:L6kernel}
We say that $\lambda\in\mathbb{R}$ is an $L^6$-eigenvalue of a self-adjoint Schr\"{o}dinger operator $\+{\oper{A}}:H^2(\mathbb{R}^n;\mathbb{R}^m)\subset L^2(\mathbb{R}^n;\mathbb{R}^m)\to L^2(\mathbb{R}^n;\mathbb{R}^m)$ given by $\+{\oper{A}}=-\+\Delta+\+{\oper{K}}$, if there exists a function $0\neq\+u^\lambda\in   H^2_{loc}(\mathbb{R}^n;\mathbb{R}^m)\cap L^6(\mathbb{R}^n;\mathbb{R}^m)$, with $\+\nabla \+u^\lambda\in L^2(\mathbb{R}^n;\mathbb{R}^{m})^n$, such that $\+{\oper{A}}\+u^\lambda=\lambda\+ u^\lambda$ in the sense of distributions. The function $\+ u^\lambda$ is called an $L^6$-eigenfunction.
\end{definition}

\begin{remark}\label{rek:L6}
We remark that $L^6$ is a natural space to consider in three dimensions due to the compact embedding $H^1(\Omega)\hookrightarrow L^6(\Omega)$ where $\Omega\subset\mathbb{R}^3$ is a bounded and smooth domain. In fact, any function which decays at infinity and whose first derivatives are square integrable, also belongs to $L^6(\mathbb{R}^3)$. Therefore this is a natural condition for the potential formulation of Maxwell's equations where there is no physical reason for the potentials to be square integrable, but the condition that the fields are square integrable corresponds to the physical condition that the electromagnetic fields have finite energy.
\end{remark}

The proofs of these two theorems appear in \autoref{sec:finding kernel 1.5d} and \autoref{sec:finding kernel 3d}, receptively.
Let us describe the main ideas of the proofs. For brevity, we omit the $\pm$ signs distinguishing between positively and negatively charged particles in this paragraph. Since we are interested in linear instability, we linearise the Vlasov equation around $f^0$. The only nonlinear term is the forcing term $\mathbf F\cdot\nabla_vf$, so that the linearisation of \eqref{eq:vlasov} becomes
	\begin{equation}\label{eq:lin-vlasov}
	\frac{\partial f}{\partial t}+\+{\hat{v}}\cdot\nabla_{\+x} f+\mathbf F^0\cdot\nabla_{\+v}f=-\mathbf F\cdot\nabla_{\+v}f^0
	\end{equation}
where $\mathbf F^0$ is the equilibrium self consistent Lorentz force and $\mathbf F$ is the linearized Lorentz force. We make the following growing-mode ansatz:
	\begin{equation}\label{ansatz}
	\begin{split}
	\text{\textbf{Ansatz}: }&\text{the perturbations }(f,\mathbf{E},\mathbf{B})\text{ have}\\
	&\text{time dependence }e^{\lambda t},\text{ where }\lambda>0.
	\end{split}
	\end{equation}
Equation \eqref{eq:lin-vlasov} can therefore be written as
	\begin{equation}\label{eq:lin-vlasov2}
	\left(\lambda+\oper{D}\right)f=-\mathbf F\cdot\nabla_{\+v}f^0
	\end{equation}
where
	\begin{equation}\label{eq:transport-op}
	\oper{D}=\+{\hat{v}}\cdot\nabla_{\+x}+\mathbf F^0\cdot\nabla_{\+v}
	\end{equation}
is the linearised Vlasov transport operator. We then invert \eqref{eq:lin-vlasov2} by applying $\lambda\left(\lambda+\oper{D}\right)^{-1}$, which is an ergodic averaging operator along the trajectories of $\oper{D}$ (depending upon $\lambda$ as a parameter), see \cite[Eq. (2.10)]{Ben-Artzi2011}. Hence we obtain an expression of $f$ in terms of a certain average of the right hand side $-\mathbf F\cdot\nabla_{\+v}f^0$ depending upon the parameter $\lambda$ (see \eqref{f in 1.5d} and \eqref{f in 3d}). This expression for $f$ is substituted into Maxwell's equations through the particle and charge densities, resulting in a system of (elliptic) equations for the spatial variable alone (recall that the momentum variable is integrated in the expressions for $\rho$ and $\mathbf{j}$). The number of linearly independent equations is less than one would expect due to the imposed symmetries. However, in both cases the equations can be written so that they form a self-adjoint system denoted $\mathcal \M$ (see \eqref{Formally definition of M} for the $1.5d$ case and \eqref{Formally definition of M 2} for the $3d$ case) that has the general form
	$$
	\M
	=
	\begin{bmatrix}
    -\Delta+1&0\\
    0&\Delta-1
   \end{bmatrix}
   +
	\+{\oper{K}}^\lambda
	$$
acting on the electric and magnetic potentials, where $\+{\oper{K}}^\lambda$ is a uniformly bounded and  symmetric family.

The problem then reduces to showing that the equation $\M \+u=0$ has a nontrivial solution for some value of $\lambda>0$. The difficulty here is twofold: first, the spectrum of $\M$ is unbounded (not even semi-bounded) and includes essential spectrum extending to both $+\infty$ and $-\infty$. Second, for each $\lambda$, the operator $\M$ has a different spectrum: one must analyse a family of spectra that depends upon the parameter $\lambda$. In \cite{Ben-Artzi2013e} we address the following related problem:

\begin{problem}
Consider the family of self-adjoint unbounded operators
	\begin{equation*}
	\M=\+{\oper{A}}+\+{\oper{K}}^\lambda=
	\begin{bmatrix}
	-\Delta+1&0\\
	0&\Delta-1
	\end{bmatrix}
	+
	\begin{bmatrix}
	\oper{K}^\lambda_{++}&\oper{K}^\lambda_{+-}\\
	\oper{K}^\lambda_{-+}&\oper{K}^\lambda_{--}
	\end{bmatrix},\quad \lambda\in [0,1]
	\end{equation*}
acting in $L^2(\mathbb{R}^d)\oplus L^2(\mathbb{R}^d)$, where  $\{\+{\mathcal{K}}^\lambda\}_{\lambda\in [0,1]}$ is a uniformly bounded, symmetric and strongly continuous family. {Is it possible to construct explicit finite-dimensional symmetric approximations of $\M$ whose spectrum in $(-1,1)$ converges to that of $\M$ for all $\lambda$ simultaneously?}
\end{problem}

A solution to this problem allows us to construct finite-dimensional approximations to $\M$. We discuss this problem in \autoref{Abstract Section}. The conditions \eqref{eq:thm1-condition} and  \eqref{eq:thm2-condition} appearing in the main theorems above translate into analogous conditions on the approximations, and those, in turn, guarantee the existence of a nontrivial \emph{approximate} solution. Since the approximate problems converge (in an appropriate sense) to the original problem, this is enough to complete the proof. A crucial ingredient is the self-adjointness of all operators: this guarantees that the spectrum is restricted to the real line. The strategy is to ``track'' eigenvalues as a function of the parameter $\lambda$ and conclude that they cross through $0$ for some value $\lambda>0$. To do so, we require knowledge of the spectrum of the operator $\M$ for small positive $\lambda$, which is obtained from the assumptions \eqref{eq:thm1-condition} and \eqref{eq:thm2-condition}, and for large $\lambda$ which arises naturally from the form of the problem.

Yet even with a solution to this problem at hand, some difficulties remain. In the cylindrically symmetric case there is a geometric difficulty. Namely, cylindrical symmetries must be respected, a fact that requires a somewhat more cumbersome functional setup. In particular, the singular nature of the coordinate chart along the axis of symmetry requires special attention. To circumvent this issue we shall do all computation in Cartesian coordinates, and use carefully chosen subspaces to decompose the magnetic potential. The second difficulty is the lack of a spectral gap, which is due to the unbounded nature of the problem in physical space. As a consequence, the dependence of the spectrum of $\M$ on $\lambda$ is delicate, especially as $\lambda\to0$, and needs careful consideration.

\subsection{Previous results}
\paragraph{Existence theory.}
The main difficulty in attaining existence results for Vlasov systems is in controlling particle acceleration due to the nonlinear forcing term. Hence existence and uniqueness has only been proved under various symmetry assumptions. In \cite{Glassey1990} global existence in the $1.5d$ case was established and in \cite{Glassey1988a} the cylindrically symmetric case was considered. Local existence and uniqueness is due to \cite{Wollman1984}.

\paragraph{Stability theory.}
One of the important early results on (linear) stability of plasmas is that of Penrose  \cite{Penrose1960}. Two notable later results are \cite{Marsden1985,Marchioro1986}. We refer to \cite{Ben-Artzi2011} for additional references.
The current result continues a program initiated by Lin and Strauss \cite{Lin2007,Lin2008} and continued by the first author \cite{Ben-Artzi2011,Ben-Artzi2011b}. In \cite{Lin2007,Lin2008} the equilibria were always assumed to be strongly monotone, in the sense of \eqref{eq:monotone}.
This added sign condition (which is widely used within the physics community, and is believed to be crucial for stability results) allowed them to obtain in \cite{Lin2007} a linear \emph{stability} criterion which was complemented by a linear \emph{instability} criterion in \cite{Lin2008}. Combined, these two results produced a necessary and sufficient criterion for stability in the following sense: there exists a Schr\"odinger operator $\mathcal L^0$ acting only in the spatial variable, such that $\mathcal L^0\geq0$ implies spectral stability, and $\mathcal L^0\ngeq0$ implies spectral instability. In \cite{Ben-Artzi2011,Ben-Artzi2011b} the monotonicity assumption was removed, which mainly impacted the ability to obtain stability results. The instability results are similar to the ones of Lin and Strauss, though the author only considers the $1.5d$ case with periodicity. This is due to his methods which crucially require a Poincar\'e inequality. We remark that our results recover all previous results when one restricts to the monotone case.


\subsection{The \texorpdfstring{$1.5d$}{1.5d} case}
First we treat the so-called $1.5d$ case, which is the lowest dimensional setting that permits nontrivial electromagnetic fields. In this setting, the plasma is assumed to have certain symmetries in phase-space that render the distribution function to be a function of only one spatial variable $x$ and two momentum variables $\+v=(v_1,v_2)$, with $v_1$ being aligned with $x$. The only non-trivial components of the fields are the first two components of the electric field and the third component of the magnetic field: $\mathbf E=(E_1,E_2,0)$ and $\mathbf B=(0,0,B)$, and similarly for the equilibrium fields. The RVM system becomes the following system of scalar equations
		\begin{subequations}\label{eq:rvm-1.5}
		\begin{align}
		&\partial_tf^\pm+\hat{v}_1\partial_xf^\pm\pm(E_1+\hat{v}_2B)\partial_{v_1}f^\pm\pm(E_2-\hat{v}_1B)\partial_{v_2}f^\pm
		=
		0						\label{eq:vlasov-1.5}\\
		&\partial_tE_1
		=
		-j_1						\label{eq:ampere1-1.5}\\
		&\partial_tE_2+\partial_xB
		=
		-j_2						\label{eq:ampere2-1.5}\\
		&\partial_xE_1
		=
		\rho						\label{eq:gauss-1.5}\\
		&\partial_tB
		=
		-\partial_xE_2					\label{eq:faraday-1.5}
		\end{align}
		\end{subequations}
where $\rho$ and $j_1,j_2$ are defined by \eqref{eq:rho} and \eqref{eq:j}.

\subsubsection{Equilibrium}
In \autoref{sec:examples} we prove the existence of equilibria $f^{0,\pm}(x,\+v)$,  which can moreover be written as functions of the energy and momentum
	\begin{equation}\label{eq:jeans theorem ansatz}
	f^{0,\pm}(x,\+v)=\mu^\pm(e^\pm,p^\pm)
	\end{equation}
as in \eqref{eq:jeans}, where the energy and momentum are defined as:
	\begin{equation}\label{eq:constants-1.5}
	e^\pm
	=
	\left<\+v\right>\pm\phi^0(x)\pm\phi^{ext}(x),\qquad
	p^\pm
	=
	v_2\pm\psi^0(x)\pm\psi^{ext}(x)
	\end{equation}
and where $\left<\+v\right>=\sqrt{1+|\+v|^2}$, and $\phi^0$ and $\psi^0$ are the equilibrium electric and magnetic potentials (both scalar, in this case), respectively:
	\begin{equation}\label{eq:potentials-1.5}
	\partial_x\phi^0
	=
	-E_1^0,
	\qquad
	\partial_x\psi^0
	=
	B^0
	\end{equation}
and similarly $\phi^{ext}$ and $\psi^{ext}$ are external electric and magnetic potentials that give rise to external fields $E_1^{ext}$ and $B^{ext}$. It is a straightforward calculation to verify that $e^\pm$ and $p^\pm$ are conserved quantities of the Vlasov flow, i.e. that $\D e^\pm=\D p^\pm=0$, where the operators $\D$ are defined below, \eqref{eq:D}. 

\subsubsection{Linearisation}
Let us discuss the linearisation of  \eqref{eq:rvm-1.5} about a steady-state solution $(f^{0,\pm},\mathbf{E}^0,\mathbf{B}^0)$. Using ansatz \eqref{ansatz} and Jeans' theorem \eqref{eq:jeans}, together with \eqref{eq:constants-1.5} and \eqref{eq:potentials-1.5} the linearised system becomes:
		\begin{subequations}\label{RVM 1.5D linearised}
		\begin{align}
		&(\lambda+\oper{D}_\pm)f^\pm=\mp\mu_e^\pm\hat{v}_1E_1\pm\mu_p^\pm\hat{v}_1 B\mp(\mu_e^\pm\hat{v}_2+\mu_p^\pm)E_2						\label{RVM 1.5D linearised1}\\
		&\lambda E_1
		=
		-j_1						\label{RVM 1.5D linearised2}\\
		&\lambda E_2+\partial_xB
		=
		-j_2						\label{RVM 1.5D linearised3}\\
		&\partial_xE_1
		=
		\rho							\label{RVM 1.5D linearised5}\\
		&\lambda B
		=
		-\partial_xE_2,					\label{RVM 1.5D linearised4}
		\end{align}
		\end{subequations}
where
	\begin{equation}\label{eq:D}
	 \oper{D}_{\pm}=\hat{v}_1\partial_x\pm(E_1^0+E_1^{ext}+\hat{v}_2 (B^0+B^{ext}))\partial_{v_1}\pm(E_2^0-\hat{v}_2 (B^0+B^{ext}))\partial_{v_2}
	\end{equation}
are the linearised transport operators as in \eqref{eq:transport-op}, and
	\begin{equation}\label{eq:rho and j}
	\rho=\int (f^+-f^-)\,d\+v,\qquad j_i=\int\hat{v}_i(f^+-f^-)\,d\+v
	\end{equation}
are the charge and current densities, respectively.

We now construct electric and magnetic potentials $\phi$ and $\psi$, respectively, as in \eqref{eq:potentials-1.5}. Equation \eqref{RVM 1.5D linearised2} implies that $E_1$ has the same spatial support as $j_1$ which is a moment of $f^\pm$ which, in turn, has the same $x$ support as $\mu^\pm$ (this can be seen from \eqref{RVM 1.5D linearised1} for instance). We deduce that $E_1$ is  compactly supported in $\Omega$ and choose an electric potential $\phi\in H^2(\Omega)$ such that $E_1=-\partial_x\phi$ in $\Omega$ and $E_1=0$ outside $\Omega$. Since $E_1$ vanishes at the boundary of $\Omega$, we must impose Neumann boundary conditions on $\phi$, and as $E_1$ depends only on the derivative of $\phi$ we may impose that $\phi$ has mean zero. The magnetic potential $\psi$ is chosen to satisfy $B=\partial_x\psi$ and $E_2=-\lambda\psi$ (this is due to \eqref{RVM 1.5D linearised4}). Then the remaining Maxwell's equations \eqref{RVM 1.5D linearised2}-\eqref{RVM 1.5D linearised5} become
	\begin{subequations}\label{Maxwell's equations 1.5d}
	\begin{align}
	 &\lambda\partial_x\phi=-\lambda E_1=j_1\qquad&\text{ in $\Omega$}				\label{Ampere's equation 2 1.5d}\\
	&(-\partial_x^2+\lambda^2)\psi=-\partial_xB-\lambda E_2=j_2\qquad&\text{ in $\mathbb{R}$} 		\label{Ampere's equation 1 1.5d}\\
	&-\partial_x^2\phi=\partial_xE_1=\rho\qquad&\text{ in }\Omega						\label{Gauss's equation 1.5d}
	\end{align}
	\end{subequations}
where \eqref{Gauss's equation 1.5d} is complemented by the Neumann boundary condition
	\[
	-\partial_x\phi=E_1=0\text{ on }\partial\Omega.
	\]
The linearised Vlasov equations can now be written as
\begin{equation}\label{linearised Vlasov 1.5d potentials}
\begin{aligned}
(\lambda+\D)f^\pm&=\pm\mu^\pm_e\hat{v}_1\partial_x\phi\pm\mu_p^\pm\hat{v}_1\partial_x\psi
\pm\lambda(\mu_e^\pm\hat{v}_2+\mu_p^\pm)\psi\\
&=\pm\mu^\pm_e\D\phi\pm\mu_p^\pm\D\psi\pm\lambda(\mu_e^\pm\hat{v}_2+\mu_p^\pm)\psi 
\end{aligned}
\end{equation}
where we have used the fact  that $\D u=\hat{v}_1\partial_x u$ if $u$ is a function of $x$ only.

Now let us specify the functional spaces that we shall use. For the scalar potential $\phi$ we define the space
	\begin{equation*}
	L^2_0(\Omega):=\left\{f\in L^2(\Omega)\ :\ \int_\Omega f=0\right\}
	\end{equation*}
while for the magnetic potential $\psi$ we simply use $L^2(\mathbb{R})$, the standard space of square integrable functions. We denote by $H^k(\mathbb{R})$  (resp. $H^k(\Omega)$) the usual Sobolev space of functions whose first $k$ derivatives are in $L^2(\mathbb{R})$  (resp. $L^2(\Omega)$). Moreover, we naturally define 
	\begin{equation*}
	H^k_0(\Omega):=\left\{f\in H^k(\Omega)\ :\ \int_\Omega f=0\right\}
	\end{equation*}
and the corresponding version which imposes Neumann boundary conditions
	\begin{equation*}
	H^k_{0,n}(\Omega):=\left\{f\in H^k_0(\Omega)\ :\ \partial_xf=0\text{ on }\partial\Omega\right\}.
	\end{equation*}
Finally, to allow us to consider functions that do not decay at infinity we use the conditions \eqref{eq:weight} and \eqref{eq:int-condition} to define weighted spaces $\Ls_\pm$ as follows: we take the closure of the smooth and compactly supported functions of $(x,\+v)$ (with the $x$ support contained in $\Omega$) under the weighted-$L^2$ norm given by
\begin{align*}
\norm{u}_{\Ls_\pm}^2=\int_{\Omega\times\mathbb{R}^2} w^\pm|u|^2\,d\+vdx
\end{align*}
and we denote the inner product by $\ip{\cdot}{\cdot}_{\Ls_\pm}$. In particular we can view any function $u(x)\in L^2(\Omega) $ or $L^2_0(\Omega)$ as being in $\Ls_\pm$ by considering $u$ as a function of $(x,\+v)$ which does not depend on $\+v$. We can extend this to functions in $L^2(\mathbb{R})$ by multiplying them by the characteristic function of the set $\Omega$, denoted $\mathbbm{1}_\Omega$. Hence the function $\mathbbm{1}_\Omega$ itself may be regarded as an element in $\Ls_\pm$.

\subsubsection{The operators}
Finally, we define the operators used in the statement of \autoref{theorem1}. First define the following projection operators:
\begin{definition}[Projection operators]\label{def:projection operators 1.5d}
We define $\oper{Q}_\pm^0$ to be the orthogonal projection operators in $\oper{L}_\pm$ onto $\ker(\D)$.
\end{definition}
\begin{remark}\label{rem:Q0-does-not-depend-upon-weight}
Although this definition makes reference to the spaces $\Ls_\pm$, the operators $\Q[0]$ do not depend on the exact choice of weight functions $w^\pm$. This may be seen by writing $(\Q[0]h)(x,\+v)$ as the pointwise limit of ergodic averages along trajectories (see \autoref{rem:previous-definitions-of-Q} and \autoref{properties of Q}).
\end{remark}
This allows us to define the following operators acting on functions of the spatial variable $x$, not the full phase-space variables:
	\begin{subequations}\label{the 0 operators 1.5d}
	\begin{flalign}
	\oper{A}_1^0h
	=
	-\partial_x^2h+\int\sum_{\pm}\mu_e^\pm(\oper{Q}_\pm^0-1)h\,d\+v&&
	\end{flalign}
	\begin{flalign}
	\oper{A}_2^0h
	=
	-\partial_x^2h-\left(\sum_{\pm}\int\mu_p^\pm\hat{v}_2\,d\+v\right)h-\int\sum_{\pm}\hat{v}_2\mu_e^\pm\oper{Q}_\pm^0~[\hat{v}_2h]\,d\+v&&
	\end{flalign}
	\begin{flalign}
	\oper{B}^0h
	=
	\left(\int\sum_\pm\mu_p^\pm\,d\+v\right)h+\int\sum_\pm\mu_e^\pm\oper{Q}_\pm^0~[\hat{v}_2h]\,d\+v&&
	\end{flalign}
	\begin{flalign}
	\left(\oper{B}^0\right)^*h
	=
	\left(\int\sum_\pm\mu_p^\pm\,d\+v\right)h+\int\sum_\pm\mu_e^\pm\hat{v}_2\oper{Q}_\pm^0h\,d\+v.&&
	\end{flalign}
	\end{subequations}
Their precise  properties are discussed in \autoref{sec:prop operators 1.5d}. For future reference, we mention the important identity
\begin{equation}\label{eq:perfect deriv}
 \int\left(\mu^\pm_p+\hat{v}_2\mu_e^\pm\right)dv_2=0
\end{equation}
which is due to the fact that $\p{\mu^\pm}{v_2}=\mu_e^\pm\hat{v}_2+\mu_p^\pm$.

\subsection{The cylindrically symmetric case}\label{sec:intro-cyl}
Since notation can be confusing when multiple coordinate systems are in use, we start this section by making clear what our conventions are.

\paragraph{Vector transformations and notational conventions.}
We let $\+{x}=(x,y,z)=x\?e_1+y\?e_2+z\?e_3$ denote the representation of the point $\+{x}\in\mathbb{R}^3$ in terms of the standard Cartesian coordinates. We define the usual cylindrical coordinates as
	\[
	r=\sqrt{x^2+y^2},\qquad
	\theta=\arctan (y/x),\qquad
	z=z.
	\]
and the local cylindrical coordinates as
	\[
	\?e_r=r^{-1}(x,y,0),\qquad
	\?e_\theta=r^{-1}(-y,x,0),\qquad
	\?e_z=(0,0,1).
	\]
By \emph{cylindrically  symmetric} we mean that in what follows no quantity depends upon $\theta$ (which does not imply that the $\theta$ component is zero!). When writing $f(\+{x})$ we mean the value of the function $f$ at the point $\+{x}$ in Cartesian coordinates. We shall often abuse notation and write $f(r,\theta,z)$ to mean the value of $f$ at the point $(r,\theta,z)$ in cylindrical coordinates. A point $\+{v}\in\mathbb{R}^3$ in momentum space shall either be expressed in Cartesian coordinates as
	\[
	\+{v}_{xyz}=(v_x,v_y,v_z)=(\+{v}\cdot\?e_1)\?e_1+(\+{v}\cdot\?e_2)\?e_2+(\+{v}\cdot\?e_3)\mathbf{e}_3
	\]
or in cylindrical coordinates \emph{(depending upon the point $\+{x}\in\mathbb{R}^3$ in physical space)} as
	\[
	\+{v}_{rz\theta}=(v_r,v_\theta,v_z)=(\+{v}\cdot\?e_r)\?e_r+(\+{v}\cdot\?e_\theta)\?e_\theta+(\+{v}\cdot\?{e}_z)\?e_z.
	\]
However we shall not be too pedantic about this notation, and shall use $\+v$ (rather than $\+{v}_{xyz}$ or $\+{v}_{rz\theta}$) when there's no reason for confusion.

A vector-valued function $\mathbf{F}$ shall be understood to be represented in Cartesian coordinates. That is, unless otherwise said, $\mathbf{F}=(F_x,F_y,F_z)=F_x\?e_1+F_y\?e_2+F_z\?e_3$.  Its expression in cylindrical coordinates shall typically be written as $\mathbf{F}=F_r\?e_r+F_\theta\?e_\theta+F_z\?e_z$.

\paragraph{Differential operators.}
Partial derivatives in Cartesian coordinates are written as $\partial_{x},\partial_{y}$ and $\partial_z$, while in cylindrical coordinates they are $\partial_r,\partial_\theta$ and $\partial_z$. They transform in the standard manner. 
It is important to note that since we work in phase space, we shall require derivatives with respect to $\+{v}$ as well.  One important factor appearing in the Vlasov equation is $\+{\hat{v}}_{\+{x}}\cdot\nabla_{\+{x}}$, which transforms as
	\begin{align*}
	(\+{\hat{v}}_{\+{x}}\cdot\nabla_{\+{x}}) h
	&=
	\hat{v}_x\partial_x h+\hat{v}_y\partial_y h+\hat{v}_z\partial_zh\nonumber\\
	&=
	\hat{v}_r\partial_r h+r^{-1}\hat{v}_\theta\partial_\theta h+\hat{v}_z\partial_zh\\
	&=
	\hat{v}_r\partial_rh+r^{-1}\hat{v}_\theta(v_\theta\partial_{v_r}h-v_r\partial_{v_\theta}h)+\hat{v}_z\partial_zh.\nonumber
	\end{align*}
However the next term in the Vlasov equation transforms ``neatly'':
	\begin{align*}
	(\mathbf{F}\cdot\nabla_{\+{v}})h
	&=F_x\partial_{v_x}h+F_y\partial_{v_y}h+F_z\partial_{v_z}h\\
	&=(F_x\cos\theta+F_y\sin\theta)\partial_{v_r}h+(-F_x\sin\theta+F_y\cos\theta)\partial_{v_\theta}h+F_z\partial_{v_z}h\\
	&=F_r\partial_{v_r}h+F_\theta\partial_{v_\theta}h+F_z\partial_{v_z}h.
	\end{align*}

\subsubsection{The Lorenz gauge}
As opposed to the system \eqref{RVM 1.5D linearised}, here we do not get a system of scalar equations. We begin with Maxwell's equations. It is well known that there is some freedom in defining the electromagnetic potentials $\varphi$ (we use $\varphi$ in the cylindrically symmetric case rather than $\phi$ to avoid confusion) and $\?{A}$, satisfying
	\begin{equation*} 
	\partial_t\?A+\nabla\varphi=-\mathbf{E},\qquad
	\nabla\+\times\?{A}=\mathbf{B}.
	\end{equation*}

\begin{remark}\label{rek:del operator convention}
Whenever the differential operator $\nabla$ appears without any subscript, it is understood to be $\nabla_{\+{x}}$, that is, the operator $(\partial_x,\partial_y,\partial_z)$ acting on functions of the spatial variable in Cartesian coordinates. The same holds for the Laplacian $\Delta$.
\end{remark}

We choose to impose the \emph{Lorenz gauge} $\nabla\cdot\?A+\frac{\partial\varphi}{\partial t}=0$, hence transforming Maxwell's equations into the hyperbolic system
	\begin{subequations}\label{eq:maxwell-lorenz}
	\begin{align}
	&\frac{\partial^2}{\partial t^2}\varphi-\Delta\varphi=\rho,\label{eq:maxwell-lorenz-charge}\\
	&\frac{\partial^2}{\partial t^2}\?A-\+\Delta\?A=\?j.\label{eq:maxwell-lorenz-current}
	\end{align}
	\end{subequations}
We remark that this is not unique to the cylindrically symmetric case, and the expressions above are written in Cartesian coordinates.

\subsubsection{Equilibrium and the linearised system}
We define the steady-state potentials $\varphi^0:\mathbb R^3\to\mathbb R$ and $\?{A}^0:\mathbb R^3\to\mathbb R^3$ as
	\begin{equation*}\label{eq:potentials-3}
	\nabla\varphi^0=-\mathbf{E}^0,\qquad
	\nabla\times\?A^0=\mathbf{B}^0
	\end{equation*}
which become
	\begin{equation}\label{eq:potentials-32}
	\mathbf{E}^0=-\partial_r\varphi^0\mathbf{e}_r-\partial_z\varphi^0\mathbf{e}_z,\qquad
	\mathbf{B}^0=-\partial_zA_\theta^0\mathbf{e}_r+\frac{1}{r}\partial_r(rA_\theta^0)\mathbf{e}_z.
	\end{equation}
The energy and momentum may be defined (analogously to \eqref{eq:constants-1.5}) as
	\begin{equation}\label{eq:constants-3}
	\begin{split}
	&e^\pm_{cyl}
	=
	\left<\+v\right>\pm\varphi^0(r,z)\pm\varphi^{ext}(r,z),\\
	&p^\pm_{cyl}
	=
	r(v_\theta\pm A_\theta^0(r,z)\pm A_\theta^{ext}(r,z)),
	\end{split}
	\end{equation}
where we remind that $\left<\+v\right>=\sqrt{1+|{\+{v}}_{xyz}|^2}$. It is straightforward to verify that they are indeed conserved along the Vlasov flow (which is given by the integral curves of the differential operators $\Dt$, defined in \eqref{eq:dt} below). To maintain simple notation we won't insist on writing the \emph{cyl} subscript where it is clear which energy and momentum are meant. The external fields are also assumed to be cylindrically symmetric and their potentials satisfy equations analogous to \eqref{eq:potentials-32}. We recall \eqref{eq:jeans}, namely that any equilibrium is assumed to be of the form
	\begin{equation*}
	f^{0,\pm}(\+x,\+v)=\mu^\pm(e^\pm,p^\pm).
	\end{equation*}

Considering the Lorenz gauge, and applying the ansatz \eqref{ansatz} and Jeans' theorem \eqref{eq:jeans} the linearisation of the RVM system about a steady-state solution $(f^{0,\pm},\mathbf{E}^0,\mathbf{B}^0)$ is
		\begin{subequations}\label{RVM 3D linearised}
		\begin{align}
&(\lambda+\Dt)f^\pm=\pm(\lambda+\Dt)(\mu_e^\pm\varphi+r\mu_p^\pm(\?A\cdot\?e_\theta))\pm\lambda\mu_e^\pm(-\varphi+\?A\cdot\+{\hat{v}})					\label{RVM 3D linearised1}\\
		&\lambda^2\varphi-\Delta\varphi=\int(f^+-f^-)\, d\+v						\label{RVM 3D linearised2}\\
		&\lambda^2\?A-\+\Delta\?A=\int(f^+-f^-)\+{\hat{v}}\, d\+v						\label{RVM 3D linearised3}
		\end{align}
		\end{subequations}
where
	\begin{equation}\label{eq:dt}
	\begin{aligned}
	 \Dt&={\+{\hat{v}}}_{xyz}\cdot \nabla_{\+x}\pm(\?E^0+\?E^{ext}+\+{\hat{v}}_{xyz}\times(\?B^0+\?B^{ext}))\cdot\nabla_{\+v}\\
	 &=\hat{v}_r\partial_r+\hat{v}_z\partial_z+(\pm E^0_r\pm E^{ext}_r\pm\hat{v}_\theta(B^0_z+B^{ext}_z)+r^{-1}\hat{v}_\theta v_\theta)\partial_{v_r}\\
	 &\quad+(\pm\hat{v}_z(B^0_r+B^{ext}_r)\mp\hat{v}_r(B^0_z+B^{ext}_z)+r^{-1}\hat{v}_\theta v_r)\partial_{v_\theta}\\
	 &\quad\pm(E^0_z+E^{ext}_z+\hat{v}_\theta(B^0_r+B^{ext}_r))\partial_{v_z}
	\end{aligned}	
	\end{equation}

are the linearised transport operators. The Lorenz gauge condition under the growing mode ansatz is
	\begin{equation}\label{eq:lin-lorenz-gauge}
	\nabla\cdot\?A+\lambda\varphi=0.
	\end{equation}

\subsubsection{Functional spaces}
Even more so than in the $1.5d$ case, choosing convenient functional spaces is crucial, due to the singular nature of the correspondence between Cartesian and cylindrical coordinates. We define
	\begin{quote}
	$L^2_{cyl}(\mathbb{R}^3)=$  the smallest closed subspace of $L^2(\mathbb{R}^3)$ comprised of functions  which have cylindrical symmetry.
	\end{quote}
A short computation using cylindrical coordinates shows that the decomposition $L^2(\mathbb{R}^3)=L^2_{cyl}(\mathbb{R}^3)\oplus (L^2_{cyl}(\mathbb{R}^3))^\perp$ reduces the Laplacian. This means that the Laplacian commutes with the orthogonal projection of $L^2(\mathbb{R}^3)$ onto $L^2_{cyl}(\mathbb{R}^3)$. Hence the Laplacian is decomposed as
\begin{equation*}
\Delta=\Delta_{cyl}+\Delta_{{cyl}^\perp}.
\end{equation*}
As we have no use for $(L^2_{cyl}(\mathbb{R}^3))^\perp$ we shall abuse notation slightly and denote $\Delta_{cyl}$ as simply $\Delta$. We now consider vector valued functions \[\?A\in L^2_{cyl}(\mathbb{R}^3;\mathbb{R}^3):=(L^2_{cyl}(\mathbb{R}^3))^3.\] We decompose such functions as
\begin{equation}\label{eq:A-decompose}
\begin{aligned}
\?A&=(\?A\cdot\?e_\theta)\?e_\theta+((\?A\cdot\?e_r)\?e_r+(\?A\cdot\?e_z)\?e_z)\\
&=\?A_\theta+\?A_{rz}\in L^2_{\theta}(\mathbb{R}^3;\mathbb{R}^3)\oplus L^2_{rz}(\mathbb{R}^3;\mathbb{R}^3).
\end{aligned}
\end{equation}
By computing with cylindrical coordinates 
we once again discover that this reduces the vector Laplacian $\+\Delta$ on $L^2_{cyl}(\mathbb{R}^3;\mathbb{R}^3)$. Note that this reduction does not occur for $\+\Delta$ on $L^2(\mathbb{R}^3;\mathbb{R}^3)$ (i.e. without the cylindrical symmetry). 

We further define the corresponding Sobolev spaces $H^k_{cyl}(\mathbb{R}^3)$,$H^k_{\theta}(\mathbb{R}^3;\mathbb{R}^3)$, $H^k_{rz}(\mathbb{R}^3;\mathbb{R}^3)$ of functions whose first $k$ weak derivatives lie in $L^2_{cyl}(\mathbb{R}^3)$, $L^2_{\theta}(\mathbb{R}^3;\mathbb{R}^3)$ and $L^2_{rz}(\mathbb{R}^3;\mathbb{R}^3)$, respectively. Note that, because of the reductions above, $\Delta$ is self-adjoint on $L^2_{cyl}(\mathbb{R}^3)$ with domain $H^2_{cyl}(\mathbb{R}^3)$ and $\+\Delta$ is self-adjoint on each of $L^2_{\theta}(\mathbb{R}^3;\mathbb{R}^3)$ and $L^2_{rz}(\mathbb{R}^3;\mathbb{R}^3)$ with domains $H^2_{\theta}(\mathbb{R}^3;\mathbb{R}^3)$ and $H^2_{rz}(\mathbb{R}^3;\mathbb{R}^3)$, respectively.

As in the $1.5d$ case we shall require certain weighted spaces $\Ns_\pm$ that allow us to include functions that do not decay. We define $\Ns_\pm$ as the closure of the smooth compactly supported functions $u:\mathbb{R}^3\times\mathbb{R}^3\to\mathbb{R}$ which are cylindrically symmetric in the $\+x$ variable, and have $\+x$-support contained in $\Omega$, under the norms
\begin{equation*}
\norm{u}_{\Ns_\pm}=\int_{\mathbb{R}^3\times\Omega} w^\pm|u|^2\,d\+vd\+x,
\end{equation*}
where the weight functions $w^\pm$ are the ones introduced in \eqref{eq:weight}.
\subsubsection{The operators}
We now define the operators used in the statement of \autoref{theorem2}. As in the $1.5d$ case we shall require the following definition of projection operators:
\begin{definition}[Projection operators]
\label{def:projection operators 3d}
We define $\Qt[0]$ to be the orthogonal projection operators in $\Ns_\pm$ onto $\ker(\Dt)$.
\end{definition}
As in the 1.5d case, the operators $\Qt[0]$ do not depend upon the exact choice of weights $w^\pm$.
Now we are ready to define the operators of the cylindrically symmetric case. For brevity, given $\+{\hat{v}}=(\hat{v}_r,\hat{v}_\theta,\hat{v}_z)$, we define $\+{\hat{v}}_{\theta}=\hat{v}_\theta \?e_\theta$ and $\+{\hat{v}}_{rz}=\hat{v}_r\?e_r+\hat{v}_z\?e_z$. All operators act on functions of the spatial variables only: the operator $\widetilde{\oper{A}}_1^0$ acts on functions in $L^2_{cyl}(\mathbb{R}^3)$, $\widetilde{\oper{A}}_2^0$ on functions in $L^2_\theta(\mathbb{R}^3;\mathbb{R}^3)
$, $\widetilde{\oper{A}}_3^0$ on functions in $L^2_{rz}(\mathbb{R}^3;\mathbb{R}^3)
$, and $\widetilde{\oper{B}}_1^0$ on functions in $L^2_\theta(\mathbb{R}^3;\mathbb{R}^3)
$ with range $L^2_{cyl}(\mathbb{R}^3)$. We have:
\begin{subequations}\label{eq:the 0 operators 3d}
	\begin{flalign}
	\widetilde{\oper{A}}_1^0 h
	=
	-\Delta h
	+
	\int\sum_{\pm}\mu_e^\pm(\Qt[0]-1) h\,d\+v&&
	\end{flalign}
	\begin{flalign}
	\widetilde{\oper{A}}_2^0 \+h
	=
	-\+\Delta \+h
	-
	\left(r\int\sum_{\pm}\mu_p^\pm{\hat{v}}_\theta\,d\+v\right)\+h
	-
	\int\sum_{\pm}\+{\hat{v}}_\theta\mu_e^\pm\Qt[0][\+h\cdot\+{\hat{v}}_\theta]\,d\+v&&
	\end{flalign}
	\begin{flalign}
	\widetilde{\oper{A}}_3^0\+h
	=
	-\+\Delta\+h
	-
	\int\sum_{\pm}\+{\hat{v}}_{rz}\mu_e^\pm\Qt[0][\+h\cdot\+{\hat{v}}_{rz}]\,d\+v&&
	\end{flalign}
	\begin{flalign}
	\widetilde{\oper{B}}_1^0\+h
	=
	\int\sum_{\pm}\mu_e^\pm(\Qt[0]-1)[\+h\cdot\+{\hat{v}}_\theta]\,d\+v&&
	\end{flalign}
	\begin{flalign}
	(\widetilde{\oper{B}}_1^0)^*h
	=
	\int\sum_{\pm}\mu_e^\pm\+{\hat{v}}_\theta(\Qt[0]-1)h\,d\+v&&
	\end{flalign}
\end{subequations}
The precise properties of these operators are discussed in \autoref{sec:prop operators 3d}. We also mention an identity analogous to \eqref{eq:perfect deriv}
\begin{equation}\label{eq:perfect deriv 3d}
\int (r\mu_p^\pm+\hat{v}_\theta\mu^\pm_e)\,d\+v=0
\end{equation}
which is due to the integrand being a perfect derivative: $\frac{\partial\mu^\pm}{\partial v_\theta}=r\mu_p^\pm+\hat{v}_\theta\mu_e^\pm$. 
\subsection{Organization of the paper}
In \autoref{sec:background} we provide some necessary background, including the crucial result on approximating spectra found in \cite{Ben-Artzi2013e}. Then we treat the two problems -- the $1.5d$ and $3d$ cases -- in parallel: in \autoref{sec:equiv} we formulate the two problems as an equivalent family of self-adjoint problems which we then solve in \autoref{sec:solving}. The proofs of the main theorems are concluded in \autoref{sec:existence}. In \autoref{sec:properties} we provide the rigorous treatment of the various operators appearing throughout the paper, and in \autoref{sec:examples} we show that there exist nontrivial equilibria.

\section{Background, Definitions and Notation}\label{sec:background}
In this section we remind the reader of the various notions of convergence in Hilbert spaces in order to avoid confusion. For a Hilbert space $\spac{H}$ we denote its norm and inner product by $\norm{\cdot}_{\spac{H}}$ and $\ip{\cdot}{\cdot}_{\spac{H}}$, respectively. When there is no ambiguity we drop the subscript. We denote the set of bounded linear operators from a Hilbert space $\spac{H}$ to a Hilbert space $\spac{G}$ as $\bounded{\spac{H},\spac{G}}$, and when  $\spac{H}=\spac{G}$ we simply write $\bounded{\spac{H}}$. The operator norm is denoted $\norm{\cdot}_{\spac{H}\to\spac{G}}$, where, again, when there is no ambiguity we may drop the subscript.
\begin{definition}[Convergence in $\bounded{\spac{H},\spac{G}}$]
Let $\oper{T},\oper{T}_n\in\bounded{\spac{H},\spac{G}}$, where $n\in\mathbb{N}$.
\begin{enumerate}[(a)]
\item We say that the sequence $\oper{T}_n$ converges to $\oper{T}$ in norm (or uniformly) as $n\to\infty$ iff $\norm{\oper{T}_n-\oper{T}}_{\spac{H}\to\spac{G}}\to0$ as $n\to\infty$. In this case we write $\oper{T}_n\to\oper{T}$.
\item We say that the sequence $\oper{T}_n$ converges  to $\oper{T}$ strongly as $n\to\infty$ iff we have the pointwise convergence $\oper{T}_nu\to\oper{T}u$ in $\spac{H}$ for all $u\in\spac{H}$. In this case we write $\oper{T}_n\tos\oper{T}$.
\end{enumerate}
\end{definition}
Now let us recall some important notions related to unbounded self-adjoint operators:

\begin{definition}[Convergence of unbounded operators]
Let $\oper{A}$ and $\oper{A}_n$ be self-adjoint, where $n\in\mathbb{N}$.
  \begin{enumerate}[(a)]
  \item
  We say that the sequence $\oper{A}_n$ converges to $\oper{A}$ in the norm resolvent sense as $n\to\infty$ iff $(\oper{A}_n-z)^{-1}\to(\oper{A}-z)^{-1}$ for any $z\in\mathbb{C}\setminus\mathbb{R}$. In this case we write $\oper{A}_n\tonr\oper{A}$.
  \item
  We say that the sequence $\oper{A}_n$ converges to $\oper{A}$ in the strong resolvent sense as $n\to\infty$ iff $(\oper{A}_n-z)^{-1}\tos(\oper{A}-z)^{-1}$ for any $z\in\mathbb{C}\setminus\mathbb{R}$. In this case we write $\oper{A}_n\tosr\oper{A}$.
  \end{enumerate}
\end{definition}

\begin{remark}
Notice that for any self-adjoint operator $\oper{A}$, the resolvent $(\oper{A}-z)^{-1}$ is a bounded operator for any $z\in\mathbb{C}\setminus\mathbb{R}$.
\end{remark}

\subsection{Basic facts}
The subsequent results will be used throughout the paper without explicit reference.
\begin{lemma}
Let $\spac{H},\spac{G}$ be Banach spaces, $\oper{T},\oper{T}_n\in\bounded{\spac{H},\spac{G}}$ and $u,u_n\in\spac{H}$ where $n\in\mathbb{N}$, and assume that $\oper{T}_n\tos\oper{T}$ and $u_n\to u$. Then $\oper{T}_nu_n\to u$ as $n\to\infty$.
\end{lemma}
\begin{proof}
We compute
\begin{equation*}
\begin{aligned}
\norm{\oper{T}_nu_n-\oper{T}u}&\le\norm{\oper{T}_n(u_n-u)}+\norm{(\oper{T}_n-\oper{T})u}\\
&\le\left( \sup_{n\in\mathbb{N}}\norm{\oper{T}_n}\right)\norm{u_n-u}+\norm{(\oper{T}_n-\oper{T})u}
\end{aligned}
\end{equation*}
This supremum is finite by the uniform boundedness principle, so the first term converges to zero since $u_n\to u$. The second term converges to zero since $\oper{T}_n\tos\oper{T}$.
\end{proof}
\begin{corollary}
If $\oper{T}_n\tos\oper{T}$ and $\oper{S}_n\tos\oper{S}$ as $n\to\infty$ then $\oper{T}_n\oper{S}_n\tos\oper{T}\oper{S}$ as $n\to\infty$.
\end{corollary}

The following result complements Weyl's theorem (see \cite[IV, Theorem 5.35]{Kato1995}) on the stability of the essential spectrum under a relatively compact perturbation. In our setting we know more about the perturbation than being merely relatively compact.
\begin{lemma}\label{lemma:weyl-finite}
Let $\mathcal{A}=-\Delta+\mathcal{K}: H^2(\mathbb{R}^n)\subset L^2(\mathbb{R}^n)\to L^2(\mathbb{R}^n)$ be a self-adjoint Schr\"odinger operator with $\mathcal{K}\in\spac{B}(L^2(\mathbb{R}^n))$ and $\oper{K}=\oper{K}\oper{P}$ where $\oper{P}:L^2(\mathbb{R}^n)\to L^2(\mathbb{R}^n)$ is the multiplication operator by the characteristic function $\mathbbm{1}_{\Omega}$ of some bounded domain $\Omega\subset\mathbb{R}^n$. Then $\mathcal{A}$ has a finite number of negative eigenvalues.
\end{lemma}
\begin{proof}
We want to show that the set
	\begin{equation*}
	\mathcal N
	:=
	\left\{u\,:\,\left<\oper{A}u,u\right><0\right\}
	=
	\left\{u\,:\,\left<\oper{K}u,u\right><-\norm{\nabla u}^2\right\}
	\end{equation*}
is finite dimensional. Consider the larger set
	\[
	\mathcal N\subset \mathcal N':=\left\{u\,:\,\norm{\oper{K}}\norm{\oper{P}u}^2>\norm{\nabla u}^2\right\}.
	\]
We claim that it is sufficient to restrict to functions $u$ supported in $\Omega$.  Indeed, letting $\chi$ be a non-negative, compactly supported smooth cutoff function which equals $1$ in $\Omega$, we have
\begin{equation*}
\begin{aligned}
\ip{-\Delta u}{u}&=\norm{\nabla u}^2\ge \norm{\chi\nabla u}^2=\norm{[\chi,\nabla]u+\nabla(\chi u)}^2\\
&=\norm{\nabla(\chi u)}^2+\norm{[\chi,\nabla]u}^2+2\ip{[\chi,\nabla]u}{\nabla(\chi u)}\\
&\ge \norm{\nabla(\chi u)}^2-2\norm{\nabla\chi}_{L^\infty}\norm{u}\norm{\nabla(\chi u)}\\
&\ge \frac12\norm{\nabla(\chi u)}^2-C\norm{u}^2
\end{aligned}
\end{equation*}
where $[\chi,\nabla]f=\chi\nabla f-\nabla(\chi f)=(\nabla\chi)f$ is the commutator between $\nabla$ and multiplication by $\chi$. By choosing $\chi$ appropriately, the constant $C$ may be made arbitrarily small.  This proves the claim. Finally, by the boundedness of $\Omega$ the set $\mathcal N'$ is clearly finite dimensional, as there can only be finitely many energy levels below any given threshold in a bounded domain.
\end{proof}

\subsection{Approximating strongly continuous families of unbounded operators}\label{Abstract Section}
Here we summarise the main results of \cite{Ben-Artzi2013e} on properties of approximations of strongly continuous families of unbounded operators. We shall require these results in the sequel. We refer to \cite{Ben-Artzi2013e} for full details, including proofs. Let $\mathfrak{H}=\mathfrak{H}_+\oplus\mathfrak{H}_-$ be a (separable) Hilbert space with inner product $\ip{\cdot}{\cdot}$ and norm $\norm{\cdot}$ and let 
	\begin{equation*}
	\oper{A}^\lambda=
	\begin{bmatrix}
	\mathcal{A}_+^\lambda&0\\
	0&-\mathcal{A}_-^\lambda
	\end{bmatrix}
	\quad\text{and}\quad
	\oper{K}^\lambda=
	\begin{bmatrix}
	\oper{K}^\lambda_{++}&\oper{K}^\lambda_{+-}\\
	\oper{K}^\lambda_{-+}&\oper{K}^\lambda_{--}
	\end{bmatrix},\quad\lambda\in[0,1]
	\end{equation*}
be two families of operators on $\mathfrak{H}$ depending upon the parameter $\lambda\in[0,1]$ (the range $[0,1]$ of values of the parameter is, of course, arbitrary), where the family $\oper{A}^\lambda$ is also assumed to be defined for $\lambda$ in an open neighbourhood $D_0$ of $[0,1]$ in the complex plane. They satisfy:\\

\textbf{i) Sectoriality:} The sesquilinear forms $\sform{a}^\lambda_\pm$ corresponding to $\oper{A}^\lambda_\pm$ are sectorial for $\lambda\in D_0$, symmetric for real $\lambda$, have dense domains $\dom{\sform{a}^\lambda_\pm}$ independent of $\lambda\in D_0$, and $D_0\ni\lambda\mapsto\sform{a}^\lambda_\pm[u,v]$ are holomorphic for any $u,v\in\dom{\sform{a}^\lambda_\pm}$. [In the terminology of \cite{Kato1995}, $\sform{a}^\lambda_\pm$ are  \emph{holomorphic families of type (a)} and $\oper{A}^\lambda$ are  \emph{holomorphic families of type (B)}.] \\

\textbf{ii) Gap:} $\oper{A}^\lambda_\pm>1$ for every $\lambda\in[0,1]$.\\

\textbf{iii) Bounded perturbation:} $\{\oper{K}^\lambda\}_{\lambda\in[0,1]}\subset\bounded{\mathfrak{H}}$ is a symmetric strongly continuous family.\\

\textbf{iv) Compactness:} There exist symmetric operators $\oper{P}_\pm\in\bounded{\mathfrak{H}_\pm}$ which are relatively compact with respect to the forms $\sform{A}^\lambda_\pm$, satisfying $\oper{K}^\lambda=\oper{K}^\lambda\oper{P}$ for all $\lambda\in[0,1]$ where \[\oper{P}=\begin{bmatrix}\oper{P}_+&0\\0&\oper{P}_-\end{bmatrix}.\]\\

Finally, if the family $\oper{A}^\lambda$ does not have a compact resolvent we assume:\\

\textbf{v) Compactification of the resolvent:}
There exist holomorphic forms $\{\sform{w}_\pm^\lambda\}_{\lambda\in D_0}$ of type (a) and associated operators $\{\oper{W}^\lambda_\pm\}_{\lambda\in D_0}$ of type (B) such that  for $\lambda\in[0,1]$, $\oper{W}^\lambda_\pm$ are self-adjoint and non-negative, and if  $\sform{W}^\lambda$ is the form associated with
 \[\oper{W^\lambda}=\begin{bmatrix}\oper{W}_+^\lambda&0\\0&-\oper{W}_-^\lambda\end{bmatrix},\quad \lambda\in D_0,\] 
then $\dom{\sform{W}^\lambda}\cap\dom{\sform{A}_\pm}$ are dense for all $\lambda\in D_0$ and the inclusion $(\dom{\sform{W}^\lambda}\cap\dom{\sform{A}},\norm{\cdot}_{\sform{A}_\varepsilon^\lambda})\to(\mathfrak{H},\norm{\cdot})$ is compact for some $\lambda\in D_0$ and all $\varepsilon>0$, where $\sform{A}_\varepsilon^\lambda$ is the form associated with
	\begin{equation}\label{eq:a-epsilon}
	\oper{A}_\varepsilon^\lambda:=\oper{A}^\lambda+\varepsilon\oper{W}^\lambda,\quad \lambda\in D_0,\ \varepsilon\ge0.\\
	\end{equation}

Define the family of (unbounded) operators $\{\oper{M}^\lambda\}_{\lambda\in[0,1]}$, acting in $\mathfrak{H}$,  as
	\begin{equation*} 
	\oper{M}^\lambda=\oper{A}^\lambda+\oper{K}^\lambda,\quad \lambda\in [0,1].
	\end{equation*}

Our main result in \cite{Ben-Artzi2013e} is:

\begin{theorem}\label{thm:approximation}
Let $\oper{A}_\varepsilon^\lambda$ be as in \eqref{eq:a-epsilon}, and define
	\begin{equation*} 
	\mathcal{M}^\lambda_\varepsilon=\mathcal{A}^\lambda_\varepsilon+\mathcal{K}^\lambda,\quad\lambda\in[0,1].
	\end{equation*}
Let $\{e^\lambda_{\varepsilon,k}\}_{k\in\mathbb{N}}\subset\mathfrak{H}$ be a complete orthonormal set of eigenfunctions of $\mathcal{A}^\lambda_\varepsilon$, let $\mathcal{G}^\lambda_{\varepsilon,n}:\mathfrak{H}\to\mathfrak{H}$ be the orthogonal projection operators onto $\mathrm{span}(e^\lambda_{\varepsilon,1},\dots,e^\lambda_{\varepsilon,n})$ and let ${\oper{M}}^{\lambda}_{\varepsilon,n}$ be the $n$-dimensional operator defined as the restriction of ${\oper{M}}^{\lambda}_{\varepsilon}$ to $\mathcal{G}^\lambda_{\varepsilon,n}(\mathfrak{H})$. Fix $\varepsilon^*>0$, and define the function
\begin{equation*}
\begin{aligned}
 \Sigma:[0,1]\times[0,\varepsilon^*]&\to(\text{subsets of } (-1,1),d_H)\\
\Sigma(\lambda,\varepsilon)&=(-1,1)\cap\mathrm{sp}(\oper{M}^{\lambda}_{\varepsilon})\\
\end{aligned}
\end{equation*}
and for fixed $\varepsilon>0$ the function
\begin{equation*}
\begin{aligned}
 \Sigma_\varepsilon:[0,1]\times\overline{\mathbb{N}}&\to(\text{subsets of } (-1,1),d_H)\\
\Sigma_\varepsilon(\lambda,n)&=(-1,1)\cap\mathrm{sp}({\oper{M}}^{\lambda}_{\varepsilon,n})\\
\end{aligned}
\end{equation*}
where $d_H$ is the Hausdorff distance, $\overline{\mathbb{N}}=\mathbb{N}\cup\{\infty\}$ and where we use the convention that ${\oper{M}}^{\lambda}_{\varepsilon,\infty}:=\oper{M}_\varepsilon^\lambda$. Then $\Sigma$ and $\Sigma_\varepsilon$ are continuous functions of their arguments in the standard topologies on $\mathbb{R}$ (and its subsets) and $\overline{\mathbb{N}}$.
\end{theorem}
We recall the definition of the Hausdorff distance between two bounded sets $\Xi,\Upsilon\subset\mathbb{C}$:
	\[
	d_H(\Xi,\Upsilon)=\max\left(\sup_{y\in \Upsilon}\inf_{x\in \Xi}|x-y|,\sup_{x\in \Xi}\inf_{y\in \Upsilon}|x-y|\right).
	\]

\section{An equivalent problem}\label{sec:equiv}
\subsection{The \texorpdfstring{$1.5d$}{1.5d} case}
\label{Sec:equiv RVM15d}
We will now reduce the linearised Vlasov-Maxwell system \eqref{RVM 1.5D linearised} to a self-adjoint problem in $L^2(\mathbb{R})\times L^2_0(\Omega)$ depending continuously (in the norm resolvent sense) on the parameter $\lambda>0$.

\subsubsection{Inverting the linearised Vlasov equation}
Rearranging the terms in \eqref{linearised Vlasov 1.5d potentials} we obtain
\begin{equation}\label{final linearised RVM}
 (\lambda+\D)f^\pm=\pm(\lambda+\D)(\mu^\pm_e\phi+{\mu_p^\pm}\psi)\pm\lambda\mu^\pm_e(-\phi+\hat{v}_2\psi),
\end{equation}
where we use the fact that $\mu^\pm$ are constant along trajectories of the vector-fields $\D$.
In order to obtain an expression for $f^\pm$ in terms of the potentials $\phi,\psi$ we invert the operator $(\lambda+\D)$ and to do this we must study the operators $\D$.
\begin{lemma}\label{properties of D}
The operators $\D$ on $\Ls_\pm$ satisfy:
\begin{enumerate}[(a)]
\item $\D$ are skew-adjoint and the resolvents $(\lambda+\D)^{-1}$ are bounded linear operators for $\operatorname{Re}\lambda\ne0$ with norm bounded by $1/|\operatorname{Re}\lambda|$.
\item $\D$ flip parity with respect to the variable $v_1$, i.e. if $h(x,v_1,v_2)\in\dom{\D}$ is an even function of $v_1$ then $\D h$ is an odd function of $v_1$ and vice versa.
\item For real $\lambda\ne0$ the resolvents of $\D$ split as follows:
\begin{equation*}
(\lambda+\D)^{-1}=\lambda(\lambda^2-\D^2)^{-1}-\D(\lambda^2-\D^2)^{-1}
\end{equation*}
where the first part is symmetric and preserves parity with respect to $v_1$, and the second part is skew-symmetric and inverts parity with respect to $v_1$.
\end{enumerate}
\end{lemma}
\begin{proof}
Skew-adjointness follows from integration by parts, noting that $w^\pm$ are in the kernels of $\D$ (to be fully precise, only skew-\emph{symmetry} follows. However, skew-adjointness is a simple extension, see e.g. \cite{Tao2011b} and in particular exercise 28 therein). The existence of bounded resolvents follows. The statement regarding parity follows directly from the formulas for $\D$ term by term. Finally, for the last part we use functional calculus formalism to compute
\begin{equation*}
\frac{1}{\lambda+\D}=\frac{\lambda-\D}{\lambda^2-\D^2}=\frac{\lambda}{\lambda^2-\D^2}-\frac{\D}{\lambda^2-\D^2}.
\end{equation*}
As $\D$ are skew-adjoint, $\D^2$ are self-adjoint and hence the first term is self-adjoint and the second skew-adjoint. For the parity properties we note that as $\D$ flip parity, $\D^2$ preserve parity and hence so do $\lambda^2-\D^2$ and their inverses.
\end{proof}
Applying $(\lambda+\D)^{-1}$ to \eqref{final linearised RVM} yields,
\begin{equation}\label{f in 1.5d}
f^\pm=\pm\mu^\pm_e\phi\pm{\mu_p^\pm}\psi\pm\lambda(\lambda+\D)^{-1}[\mu_e^\pm(-\phi+\hat{v}_2\psi)].
\end{equation}
Furthermore, using \autoref{properties of D} we split $f^\pm$ into even and odd functions of $v_1$:
\begin{align*}
f^\pm_{ev}&=\pm\mu^\pm_e\phi\pm{\mu_p^\pm}\psi\pm\mu^\pm_e\lambda^2(\lambda^2-\D^2)^{-1}[-\phi+\hat{v}_2\psi]\\
f^\pm_{od}&=\mp\mu^\pm_e\lambda\D(\lambda^2-\D^2)^{-1}[-\phi+\hat{v}_2\psi]
\end{align*}
using the fact that $\phi,\psi$ and $\mu$ are all even functions of $v_1$. For brevity, we define operators $\Q:\Ls_\pm\to\Ls_\pm$ as
	\begin{equation*}
	\Q=\lambda^2(\lambda^2-\D^2)^{-1},\quad\lambda>0.
	\end{equation*}
When $\lambda\to0$ their strong limits exist, and are defined in \autoref{def:projection operators 1.5d} (this convergence is proved in \autoref{properties of Q}).

\begin{remark}\label{rem:previous-definitions-of-Q}
Operators $\Q$ also appeared in the prior works \cite{Lin2007,Lin2008,Ben-Artzi2011,Ben-Artzi2011b}. In each of these $\Q$ were defined as an integrated average over the characteristics of the operators $\D$. In fact, as the Laplace transform of a semigroup is the resolvent of its generator we see that the operators $\Q$ in these prior works have the rule:
\begin{equation*}
\Q h=\int_{-\infty}^0\lambda e^{\lambda s} e^{s\D}h\,ds=\lambda\int^\infty_0e^{-\lambda s}e^{-s\D}h\,ds=\lambda(\lambda+\D)^{-1}h.
\end{equation*}
Here we have defined the operators $\Q$ directly from the resolvents of $\D$ as this makes some of its properties clearer, although both approaches have advantages. In particular we are able to split $\lambda(\D+\lambda)^{-1}$ into symmetric and skew-symmetric parts in \autoref{properties of D} which simplifies some computations.
\end{remark}

\subsubsection{Reformulating Maxwell's equations}
Now we substitute the expressions \eqref{f in 1.5d} into Maxwell's equations \eqref{Maxwell's equations 1.5d}. This shall result in an equivalent system of equations for $\phi$ and $\psi$. Due to the integration $d\+v$ we notice that $f^\pm_{od}$ and $f^\pm_{od}\hat{v}_2$ both integrate to zero, so that $\rho$ and $j_2$ only depend on $f^\pm_{ev}$. 

\begin{remark}
It is important to note that due to the continuity equation it is possible to express either \eqref{Ampere's equation 2 1.5d} or \eqref{Ampere's equation 1 1.5d} using the remaining two equations in \eqref{Maxwell's equations 1.5d}. See \autoref{Recovery of Maxwell's equations 1.5d}.
\end{remark}

\paragraph{Gauss' equation \eqref{Gauss's equation 1.5d}.}
Gauss' equation becomes
	\begin{equation}\label{eq:Gauss reformulated}
	\begin{aligned}
	-\partial_x^2\phi
	&=
	\int(f^+_{ev}-f^-_{ev})\, d\+v\\
	&=\int
\sum_{\pm}\left(\mu_e^\pm\phi+\mu_p^\pm\psi+\Q{}[\mu_e^\pm(-\phi+\hat{v}_2\psi)]\right)
\,d\+v \\
&=\int\sum_{\pm}(\mu_p^\pm+\mu_e^\pm\hat{v}_2)\psi\,d\+v+\int\sum_{\pm}\mu_e^\pm(\Q-1)[-\phi+\hat{v}_2\psi]\,d\+v,
	\end{aligned}
	\end{equation}
where we have pulled $\mu_e^\pm$ outside the application of $\Q$ as they belong to $\ker(\D)$.

\paragraph{Amp\`ere's equation \eqref{Ampere's equation 1 1.5d}.}
Similarly, Amp\`ere's equation becomes
	\begin{equation}\label{eq:Ampere reformulated}
	\begin{aligned}
	(-\partial_x^2+\lambda^2)\psi
	&\!=\!
	\int\hat{v}_2(f^+_{ev}-f^-_{ev})\, d\+v\\
	&\!=\!\int
\sum_{\pm}\hat{v}_2\left(\mu_e^\pm\phi+\mu_p^\pm\psi+\Q{}[\mu_e^\pm(-\phi+\hat{v}_2\psi)]\right)
\,d\+v \\
&\!=\!\int\sum_{\pm}\hat{v}_2(\mu_p^\pm+\mu_e^\pm\hat{v}_2)\psi\,d\+v\\&\qquad+\int\sum_{\pm}\hat{v}_2\mu_e^\pm(\Q-1)[-\phi+\hat{v}_2\psi]\,d\+v.
	\end{aligned}
	\end{equation}

\paragraph{An equivalent formulation.}
We write the two new expressions \eqref{eq:Gauss reformulated} and \eqref{eq:Ampere reformulated} abstractly in the compact form
\begin{equation}\label{Formally definition of M}
\M \begin{bmatrix}
\psi\\\phi
\end{bmatrix}
= \begin{bmatrix}
-\partial_x^2\psi+\lambda^2\psi-j_2\\\partial_x^2\phi+\rho
\end{bmatrix}=\begin{bmatrix}0\\0\end{bmatrix},
\end{equation}
where, for $\lambda>0$,  $\M$ is a self-adjoint matrix of operators mapping $L^2(\mathbb R)\times L^2_0(\Omega)\to L^2(\mathbb R)\times L^2_0(\Omega)$ (see \autoref{lem:prop M}). We claim that this operator may be written either as
	\begin{equation}\label{eq:m lambda 1}
	\M
	=
	\begin{bmatrix}-\partial_x^2+\lambda^2&0\\0&\partial_x^2\end{bmatrix}-\J
	\end{equation}
or, equivalently, as
	\begin{equation}\label{eq:m lambda 2}
	\M
	=
	\begin{bmatrix}
	\oper{A}_2^\lambda&\left(\oper{B}^\lambda\right)^*\\\oper{B}^\lambda&-\oper{A}_1^\lambda
	\end{bmatrix}
	\end{equation}
where the various operators appearing above are given by
	\begin{subequations}\label{eq:operators}
	\begin{flalign}
	\begin{aligned}
		\J \begin{bmatrix}
		h\\g
		\end{bmatrix}
		&=-\left(\sum_{\pm}\int\mu^\pm\frac{1+v_1^2}{\<{\+v}^3}\,d\+v\right)\begin{bmatrix}
		h\\0
		\end{bmatrix}+\\
		&\quad+\sum_{\pm}\int\begin{bmatrix}
		\hat{v}_2\\-1
		\end{bmatrix}
		\mu_e^\pm(\Q-1)\left(\begin{bmatrix}
		\hat{v}_2\\-1
		\end{bmatrix}\cdot\begin{bmatrix}
		h\\g
		\end{bmatrix}\right)\,d\+v
		\end{aligned}\label{eq:J lambda}
	\end{flalign}	
	\begin{flalign}
	\oper{A}_1^\lambda h
	=
	-\partial_x^2 h+\int\sum_{\pm}\mu_e^\pm(\Q-1) h\,d\+v&&
	\end{flalign}
	\begin{flalign}
	\oper{A}_2^\lambda h
	=
	-\partial_x^2 h+\lambda^2 h-\left(\sum_{\pm}\int\mu_p^\pm\hat{v}_2\,d\+v\right) h-\int\sum_{\pm}\hat{v}_2\mu_e^\pm\Q~[\hat{v}_2 h]\,d\+v&&
	\end{flalign}
	\begin{flalign}
	\oper{B}^\lambda h
	=
	\left(\int\sum_\pm\mu_p^\pm\,d\+v\right) h+\int\sum_\pm\mu_e^\pm\Q~[\hat{v}_2 h]\,d\+v&&
	\end{flalign}
	\begin{flalign}
	\left(\oper{B}^\lambda\right)^* h
	=
	\left(\int\sum_\pm\mu_p^\pm\,d\+v\right) h+\int\sum_\pm\mu_e^\pm\hat{v}_2\Q h\,d\+v.&&
	\end{flalign}
	\end{subequations}

\begin{remark}
Though $\lambda>0$ in the foregoing discussion, all operators can be defined for $\lambda=0$, as we have already done for some (see \eqref{the 0 operators 1.5d}).
\end{remark}

The expression \eqref{eq:m lambda 2} is no more than a rewriting of \eqref{eq:Gauss reformulated} and \eqref{eq:Ampere reformulated}. However the expression \eqref{eq:m lambda 1} requires some attention. In particular, to obtain it one has to use \eqref{eq:perfect deriv} as well as the integration by parts
\begin{equation*}
\int\p{\mu^\pm}{v_2}\hat{v}_2\,d\+v=-\int\mu^\pm\p{\hat{v}_2}{v_2}\,d\+v=-\int\mu^\pm\frac{1+v_1^2}{\<{\+v}^3}\,d\+v.
\end{equation*}

The properties of the operators appearing in \eqref{eq:operators} are discussed in details in \autoref{properties of J} and \autoref{lem:prop A}. Let us briefly summarise:
	\begin{itemize}
	\item
	$\oper{A}_1^\lambda:H^2_{n,0}(\Omega)\subset L_0^2(\Omega)\to L_0^2(\Omega)$ is self-adjoint and has a purely discrete spectrum with finitely many negative eigenvalues.
	\item
	$\oper{A}_2^\lambda:H^2(\mathbb{R})\subset L^2(\mathbb{R})\to L^2(\mathbb{R})$ is self-adjoint, has essential spectrum in $[\lambda^2,\infty)$ and  finitely many negative eigenvalues.
	\item
	$\oper{B}^\lambda:L^2(\mathbb{R})\to L^2_0(\Omega)$ is a bounded operator, with bound independent of $\lambda$.
	\item
	$\J:L^2(\mathbb R)\times L^2_0(\Omega)\to L^2(\mathbb R)\times L^2_0(\Omega)$ is a bounded symmetric operator, with bound independent of $\lambda$.
	\end{itemize}

\subsection{The cylindrically symmetric case}
\label{Sec:equiv cyl}
Our approach here is fully analogous to the one presented in \autoref{Sec:equiv RVM15d} hence we shall keep it brief, omitting repetitions as much as possible. For convenience we denote analogous operators by the same letter, but we shall add a \emph{tilde} to any such operator in this section. Hence, e.g. the operators analogous to $\oper{D}_\pm$ shall be denoted $\widetilde{\oper{D}}_{\pm}$.

\subsubsection{Inverting the linearised Vlasov equation}
Recall the linearised Vlasov equation \eqref{RVM 3D linearised1}
\begin{equation*}
(\lambda+\Dt)f^\pm=\pm(\lambda+\Dt)(\mu_e^\pm\varphi+r\mu_p^\pm(\?A\cdot\+e_\theta))\pm\lambda\mu_e^\pm(-\varphi+\?A\cdot\+{\hat{v}}).
\end{equation*}
Inverting, we get the expression
\begin{equation}\label{f in 3d}
f^\pm=\pm\mu_e^\pm\varphi\pm r\mu_p^\pm(\?A\cdot\+e_\theta)\pm\mu_e^\pm\lambda(\lambda+\Dt)^{-1}(-\varphi+\?A\cdot\+{\hat{v}}),
\end{equation}
and, recalling that we only care about the quantity $f^+-f^-$, we write it for future reference:
\begin{equation}\label{eq:density-3d}
f^+-f^-=\sum_\pm\mu_e^\pm\varphi+\sum_\pm r\mu_p^\pm(\?A\cdot\+e_\theta)+\sum_\pm\mu_e^\pm\lambda(\lambda+\Dt)^{-1}(-\varphi+\?A\cdot\+{\hat{v}}).
\end{equation}
\begin{lemma}\label{properties of D 2}
The operators $\Dt$ on $\Ns_\pm$ satisfy:
\begin{enumerate}[(a)]
\item $\Dt$ are skew-adjoint and the resolvents $(\lambda+\Dt)^{-1}$ are bounded linear operators for $\operatorname{Re}\lambda\ne0$ with norm bounded by $1/|\operatorname{Re}\lambda|$.
\item $\Dt$ flip parity with respect to the pair of variables $(v_r,v_z)$, i.e. if $h\in\dom{\Dt}$ is an even function of the pair $(v_r,v_z)$ then $\Dt h$ is an odd function of $(v_r,v_z)$ and vice versa (see  \autoref{rek:parity} below).
\item For real $\lambda\ne0$ the resolvents of $\Dt$ split as follows:
\begin{equation}\label{eq:Dt-split}
(\lambda+\Dt)^{-1}=\lambda(\lambda^2-\Dt^2)^{-1}-\Dt(\lambda^2-\Dt^2)^{-1}
\end{equation}
where the first part is symmetric and preserves parity with respect to $(v_r,v_z)$, and the second part is skew-symmetric and inverts parity with respect to $(v_r,v_z)$.
\end{enumerate}
\end{lemma}
We leave the proof, which is analogous to the proof of \autoref{properties of D}, to the reader.
\begin{remark}\label{rek:parity}
{For a function $h$ expressed in cylindrical coordinates as $h(\+x,v_r,v_z,v_\theta)$, we say that $h$ is an even function of the pair $(v_r,v_z)$ if $h(\+x,v_r,v_z,v_\theta)=h(\+x,-v_r,-v_z,v_\theta)$, where we flip the sign of both variables simultaneously. Note that this is a weaker property than both being an even function of $v_r$ and an even function of $v_z$. Odd functions of $(v_r,v_z)$ are defined similarly.}
\end{remark}
As in the $1.5d$ case, we define averaging operators. However, in this case both the symmetric and skew-symmetric parts are required. The operators $\Qtsym$ and $\Qtskew$ map $\Ns_\pm$ to $\Ns_\pm$ and are defined by the rules
	\begin{align*}
	\Qtsym&=\lambda^2(\lambda^2-\Dt^2)^{-1},\quad\lambda>0\\
	\Qtskew&=-\lambda\Dt(\lambda^2-\Dt^2)^{-1},\quad\lambda>0.
	\end{align*}
Note that by \eqref{eq:Dt-split} we have $\lambda(\lambda+\Dt)^{-1}=\Qtsym + \Qtskew$.
\subsubsection{Reformulating Maxwell's equations}
We now rewrite Maxwell's equations \eqref{RVM 3D linearised2}-\eqref{RVM 3D linearised3} as an equivalent self-adjoint problem using the expression \eqref{eq:density-3d}. We start with \eqref{RVM 3D linearised2}:
	\begin{equation}\label{eq:Gauss reformulated 3d}
	\begin{aligned}
	0&=\lambda^2\varphi-\Delta\varphi-\int(f^+-f^-)\, d\+v\\
	&=\lambda^2\varphi-\Delta\varphi-\int\sum_{\pm}\left(\mu_e^\pm\varphi+ r\mu_p^\pm(\?A\cdot\+e_\theta)+\mu_e^\pm\lambda(\lambda+\Dt)^{-1}(-\varphi+\?A\cdot\+{\hat{v}})\right)\, d\+v
	\end{aligned}
	\end{equation}
where $\varphi\in\spac{H}_\varphi$. Next, the system of equations \eqref{RVM 3D linearised3} becomes
	\begin{equation}\label{eq:Ampere reformulated 3d}
	\begin{aligned}
	0&=\lambda^2\?A-\+\Delta\?A-\int(f^+-f^-)\+{\hat{v}}\, d\+v\\
	&=\lambda^2\?A-\+\Delta\?A-\int\sum_{\pm}\left(\mu_e^\pm\varphi+ r\mu_p^\pm(\?A\cdot\+e_\theta)+\mu_e^\pm\lambda(\lambda+\Dt)^{-1}(-\varphi+\?A\cdot\+{\hat{v}})\right)\+{\hat{v}}\, d\+v,
	\end{aligned}
	\end{equation}
where $\?A=(\?A_\theta,\?A_{rz})\in L^2_{\theta}(\mathbb{R}^3;\mathbb{R}^3)\times L^2_{rz}(\mathbb{R}^3;\mathbb{R}^3)$. 
As in \eqref{Formally definition of M}, we shall write these equations as a single system of the form
\begin{equation}\label{Formally definition of M 2}
\Mt \begin{bmatrix}
\?A_\theta\\\varphi\\\?A_{rz}
\end{bmatrix}
=\begin{bmatrix}0\\0\\0\end{bmatrix},
\end{equation}
which is a self-adjoint operator in $\spac{H}$, see \autoref{lem:prop mt}. In analogy with \eqref{eq:m lambda 2}, we define
	\begin{equation}\label{eq:m lambda 2 3d}
	\Mt
	=
	\begin{bmatrix}
    \widetilde{\oper{A}}_2^\lambda&(\widetilde{\oper{B}}_1^\lambda)^*&(\widetilde{\oper{B}}_2^\lambda)^*\\
    \widetilde{\oper{B}}^\lambda_1&-\widetilde{\oper{A}}_1^\lambda&-(\widetilde{\oper{B}}_3^\lambda)^*\\
    \widetilde{\oper{B}}_2^\lambda&-\widetilde{\oper{B}}_3^\lambda&-\widetilde{\oper{A}}_3^\lambda
   \end{bmatrix}.
	\end{equation}
With $\+{\hat{v}}=(\hat{v}_r,\hat{v}_\theta,\hat{v}_z)$, we recall the notation $\+{\hat{v}}_{\theta}=\hat{v}_\theta\+e_\theta$ and $\+{\hat{v}}_{rz}=\hat{v}_r\+e_r+\hat{v}_z\+e_z$ introduced before. Then the components of $\Mt$ are now given by
	\begin{subequations}\label{eq:operators3d}
	\begin{flalign}
	\widetilde{\oper{A}}_1^\lambda h
	=\!
	-\Delta h+\lambda^2h+\int\sum_{\pm}\mu_e^\pm(\Qtsym-1) h\,d\+v&&
	\end{flalign}
	\begin{flalign}
	\widetilde{\oper{A}}_2^\lambda \+h
	=
	-\+\Delta \+h+\lambda^2 \+h
	\!-\!
	\left(\!r\!\int\sum_{\pm}\mu_p^\pm{\hat{v}}_\theta\,d\+v\right)\!\+h
	-\!
	\int\sum_{\pm}\+{\hat{v}}_\theta\mu_e^\pm\Qtsym~\!\![\+h\cdot\+{\hat{v}}_\theta]\,d\+v&&
	\end{flalign}
	\begin{flalign}
	\widetilde{\oper{A}}_3^\lambda\+h
	=
	-\+\Delta\+h+\lambda^2\+h-\int\sum_{\pm}\+{\hat{v}}_{rz}\mu_e^\pm\Qtsym{}[\+h\cdot\+{\hat{v}}_{rz}]\,d\+v&&
	\end{flalign}
	\begin{flalign}
	\widetilde{\oper{B}}_1^\lambda\+h
	=
	\int\sum_{\pm}\mu_e^\pm(\Qtsym-1)[\+h\cdot\+{\hat{v}}_\theta]\,d\+v&&
	\end{flalign}
	\begin{flalign}
	(\widetilde{\oper{B}}_1^\lambda)^*h
	=
	\int\sum_{\pm}\mu_e^\pm\+{\hat{v}}_\theta(\Qtsym-1)h\,d\+v&&
	\end{flalign}
	\begin{flalign}
	\widetilde{\oper{B}}_2^\lambda\+h
	=
	\int\sum_{\pm}\mu_e^\pm \+{\hat{v}}_{rz}\Qtskew{}[\+h\cdot\+{\hat{v}}_\theta]\,d\+v&&
	\end{flalign}
	\begin{flalign}
	(\widetilde{\oper{B}}_2^\lambda)^*\+h
	=
	-\int\sum_{\pm}\mu_e^\pm\+{\hat{v}}_\theta\Qtskew~[\+{\hat{v}}_{rz}\cdot\+h]\,d\+v&&
	\end{flalign}
	\begin{flalign}
	\widetilde{\oper{B}}_3^\lambda h
	=
	\int\sum_{\pm}\mu_e^\pm\+{\hat{v}}_{rz}\Qtskew~h\,d\+v&&
	\end{flalign}
	\begin{flalign}
	(\widetilde{\oper{B}}_3^\lambda)^* \+h
	=
	-\int\sum_{\pm}\mu_e^\pm\+{\hat{v}}_{rz}\Qtskew~[\+h\cdot\+{\hat{v}}_{rz}]\,d\+v.&&
	\end{flalign}
	\end{subequations}
These are derived from \eqref{eq:Gauss reformulated 3d} and \eqref{eq:Ampere reformulated 3d}, where some terms vanish due to parity in $(v_r,v_z)$, (see \autoref{properties of D 2}(c)). In particular, in every occurrence of $\lambda(\lambda+\Dt)^{-1}=\Qtsym+\Qtskew$, exactly one of these operators vanishes after integration $d\+v$. In addition, we have made use of \eqref{eq:perfect deriv 3d}. We further define an operator $\Jt$ as
\begin{equation*}
\Jt=\begin{bmatrix}
		\lambda^2-\+\Delta&0&0\\
		0&-\lambda^2+\Delta&0\\
		0&0&-\lambda^2+\+\Delta\\
	\end{bmatrix}
	-
	\Mt.
\end{equation*}

Let us briefly discuss these operators in further detail (their precise properties are treated in \autoref{sec:prop op cyl}):
	\begin{itemize}
	\item The operators
	\begin{align*}
	\widetilde{\oper{A}}_1^\lambda&:H^2_{cyl}(\mathbb{R}^3)\subset L^2_{cyl}(\mathbb{R}^3)\to L^2_{cyl}(\mathbb{R}^3)\\
	\widetilde{\oper{A}}_2^\lambda&:H^2_{\theta}(\mathbb{R}^3;\mathbb{R}^3)\subset L^2_{\theta}(\mathbb{R}^3;\mathbb{R}^3)\to L^2_{\theta}(\mathbb{R}^3;\mathbb{R}^3)\\
	\widetilde{\oper{A}}_3^\lambda&:H^2_{rz}(\mathbb{R}^3;\mathbb{R}^3)\subset L^2_{rz}(\mathbb{R}^3;\mathbb{R}^3)\to L^2_{rz}(\mathbb{R}^3;\mathbb{R}^3)
	\end{align*} 
	are self-adjoint, have essential spectrum in $[\lambda^2,\infty)$ and a finite number of eigenvalues in $(-\infty,\lambda^2)$.
	\item The operators
	\begin{align*}
	\widetilde{\oper{B}}_1^\lambda&:L^2_\theta(\mathbb{R}^3;\mathbb{R}^3)\to L^2_{cyl}(\mathbb{R}^3)\\	\widetilde{\oper{B}}_2^\lambda&:L^2_\theta(\mathbb{R}^3;\mathbb{R}^3)\to L^2_{rz}(\mathbb{R}^3;\mathbb{R}^3)\\
	\widetilde{\oper{B}}_3^\lambda&:L^2_{cyl}(\mathbb{R}^3)\to L^2_{rz}(\mathbb{R}^3;\mathbb{R}^3)
	\end{align*}
	are bounded, with bound independent of $\lambda$.
	\item $\Jt : L^2_\theta(\mathbb{R}^3;\mathbb{R}^3)\times L^2_{cyl}(\mathbb{R}^3)\times L^2_{rz}(\mathbb{R}^3;\mathbb{R}^3)\to L^2_\theta(\mathbb{R}^3;\mathbb{R}^3)\times L^2_{cyl}(\mathbb{R}^3)\times L^2_{rz}(\mathbb{R}^3;\mathbb{R}^3)$ is a bounded symmetric operator with bound independent of $\lambda$.
	\end{itemize}

\section{Solving the equivalent problem}\label{sec:solving}
The problem is now reduced to finding some $\lambda\in(0,\infty)$ for which the operators $\M$ (in the $1.5d$ case) and $\Mt$ (in the cylindrically symmetric case) have non-trivial kernels (not the same $\lambda$ in both cases, of course). Our method is to compare their spectrum for $\lambda=0$ and $\lambda$ very large, and use spectral continuity arguments to deduce that as $\lambda$ varies an eigenvalue must cross through $0$ (we shall show that  both operators are self-adjoint, see \autoref{lem:prop M} and \autoref{lem:prop mt} below, hence the spectrum lies on the real axis).

\subsection{The \texorpdfstring{$1.5d$}{1.5d} case}\label{sec:finding kernel 1.5d}

\subsubsection{Continuity of the spectrum at \texorpdfstring{$\lambda=0$}{lambda=0}}\label{subsec:Moving instability criterion}
Recall the condition \eqref{eq:thm1-condition} which we require for instability:
\begin{equation}\label{eigcountcond1}
\negeig(\oper{A}^0_2+(\oper{B}^0)^*(\oper{A}^0_1)^{-1}\oper{B}^0)>\negeig(\oper{A}^0_1).
\end{equation}
We wish to move this condition to values of $\lambda$ greater than $0$:
\begin{lemma}\label{eigcountcond1movedlemma}
Assume that \eqref{eigcountcond1} holds and that zero is  in the  resolvent set of $\oper{A}_1^0$. Then there exists $\lambda_*>0$ such that for all $\lambda\in[0,\lambda_*]$
\begin{equation*} 
\negeig(\oper{A}^\lambda_2+(\oper{B}^\lambda)^*(\oper{A}^\lambda_1)^{-1}\oper{B}^\lambda)>\negeig(\oper{A}^\lambda_1) .
\end{equation*}
\end{lemma}

\begin{proof}
The proof follows immediately from the following three simple steps:

\paragraph{Step 1. $\oper{A}^\lambda_1$ is invertible for small $\lambda\geq0$.} We know from \autoref{lem:prop A} (below) that $\oper{A}_1^\lambda$ is continuous in the norm resolvent sense and has discrete spectrum. The norm resolvent continuity implies that its spectrum varies continuously in $\lambda$, so as $0$ is not in its spectrum at $\lambda=0$ there exists $\lambda_*$ such that $0$ is not in the spectrum for $0\le\lambda\le\lambda_*$. Hence for all such $\lambda$, $\oper{A}^\lambda_1$ is invertible and the operator $\oper{A}^\lambda_2+(\oper{B}^\lambda)^*(\oper{A}^\lambda_1)^{-1}\oper{B}^\lambda$ is well defined.

\paragraph{Step 2.
$\negeig(\oper{A}_1^{\lambda})=\negeig(\oper{A}_1^0)$ for all $\lambda\in[0,\lambda_*]$.}
The spectrum of $\oper{A}_1^\lambda$ is purely discrete  and $0$ is in its resolvent set. This means that none of its eigenvalues can cross $0$ for small values of $\lambda$.

 \paragraph{Step 3.
 $\negeig(\oper{A}^{\lambda}_2+(\oper{B}^\lambda)^*(\oper{A}_1^{\lambda})^{-1}\oper{B}^{\lambda})\ge\negeig(\oper{A}_2^0+(\oper{B}^0)^*(\oper{A}_1^0)^{-1}\oper{B}^0)$ for all $\lambda\in[0,\lambda_*]$.}
Observe that
	\begin{itemize}
	\item
	$[0,\infty)\ni\lambda\mapsto\oper{A}^{\lambda}_2+(\oper{B}^\lambda)^*(\oper{A}_1^{\lambda})^{-1}\oper{B}^{\lambda}$ is norm resolvent continuous,
	\item
	$\oper{A}^{\lambda}_2+(\oper{B}^\lambda)^*(\oper{A}_1^{\lambda})^{-1}\oper{B}^{\lambda}$ has essential spectrum in $[\lambda^2,\infty)$,
	\item
	$\oper{A}^{\lambda}_2+(\oper{B}^\lambda)^*(\oper{A}_1^{\lambda})^{-1}\oper{B}^{\lambda}$ has finitely many negative eigenvalues.
	\end{itemize}
These statements follow from arguments similar to those appearing in the proof of \autoref{lem:prop A}(a)-(c), the last by the boundedness of the perturbation and the location of the essential spectrum (see  \autoref{lemma:weyl-finite}). Since $0$ is not in the resolvent set at $\lambda=0$ we pick $\sigma<0$ larger than all the (finitely many) negative eigenvalues of $\oper{A}^{0}_2+(\oper{B}^0)^*(\oper{A}_1^{0})^{-1}\oper{B}^{0}$. The continuous dependence of the spectrum (as a set) on the parameter $\lambda$ implies that for small values of $\lambda$ no eigenvalues cross $\sigma$ and the number of negative eigenvalues can only grow as $\lambda$ increases.
\end{proof}

\subsubsection{Truncation}
We follow the plan hinted at in \autoref{Abstract Section}: first we discretise the spectrum, then truncate. The only continuous part in the spectrum of $\M$ is due to $\oper{A}^\lambda_2$, hence we let $W(x)$ be a smooth positive potential function satisfying $W(x)\to\infty$ as $x\to\pm\infty$ which we shall add to $\oper{A}_2^\lambda$. It is well known that the Schr\"odinger operator $-\partial_x^2+W$ on $L^2(\mathbb{R})$ is self-adjoint (on an appropriate domain therein) with compact resolvent (and therefore discrete spectrum). Moreover, $C_0^\infty(\mathbb{R})$ is a core for both $\partial_x^2+W$ and $\partial_x^2$. Thus our approximating operator family is $\{\+{\oper{M}}^\lambda_{\varepsilon}\}_{\lambda\in[\lambda_*,\infty),\varepsilon\in[0,\infty)}$, where 
\begin{equation*}
\+{\oper{M}}^\lambda_\varepsilon
=
\begin{bmatrix}
    \oper{A}^\lambda_{2,\varepsilon}& (\oper{B}^\lambda)^*\\
  \oper{B}^\lambda &-\oper{A}^\lambda_{1}
   \end{bmatrix}
=
\underbrace{\begin{bmatrix}
-\partial_x^2+\varepsilon W&0\\
0&\partial_x^2
\end{bmatrix}
+
\begin{bmatrix}
\lambda^2&0\\
0&0
\end{bmatrix}}_{\+{\oper{A}}_\varepsilon^\lambda}
-\J
\end{equation*}
defined on $L^2(\mathbb{R})\times L^2_0(\Omega)$ and where $\lambda_*$ is as given in \autoref{eigcountcond1movedlemma}. For $\varepsilon>0$ this operator has discrete spectrum. As indicated  in the statement of \autoref{thm:approximation} we define truncated versions using the eigenspaces of the operator $\+{\oper{A}}_\varepsilon^\lambda$. As this operator is diagonal, we can choose the eigenvectors to lie in exactly one of $L^2(\mathbb{R})$ or $L^2_0(\Omega)$. We denote the $n$th truncation, a projection onto an eigenspace of dimension $2n$ consisting of $n$ eigenvectors in each of $L^2(\mathbb{R})$ and $L^2_0(\Omega)$, as $\+{\oper{M}}_{\varepsilon,n}^\lambda$ which is self-adjoint and defined for $\varepsilon>0,\lambda\ge0,n\in\mathbb{N}$. Moreover, the mapping $\lambda\mapsto\mathrm{sp}(\+{\oper{M}}_{\varepsilon,n}^\lambda)$ is continuous (that is, the set of eigenvalues varies continuously). In particular if there are $\lambda_*<\lambda^*$ for which $\negeig(\+{\oper{M}}_{\varepsilon,n}^{\lambda_*})\ne\negeig(\+{\oper{M}}_{\varepsilon,n}^{\lambda^*})$ then there must exist $\lambda_{\varepsilon,n}\in(\lambda_*,\lambda^*)$ for which $0\in\mathrm{sp}(\+{\oper{M}}_{\varepsilon,n}^\lambda)$. We have therefore just proved:
\begin{lemma}\label{spectral continuity argument}
Fix $\varepsilon>0,n\in\mathbb{N}$. Suppose that there exist $0<\lambda_*<\lambda^*<\infty$ such that $\negeig(\+{\oper{M}}_{\varepsilon,n}^{\lambda_*})\ne\negeig(\+{\oper{M}}_{\varepsilon,n}^{\lambda^*})$. Then there is a $\lambda_{\varepsilon,n}\in(\lambda_*,\lambda^*)$ for which $\ker(\+{\oper{M}}^\lambda_{\varepsilon,n})$ is non-trivial.
\end{lemma}

The next step is thus to establish estimates on $\negeig(\+{\oper{M}}_{\varepsilon,n}^{\lambda})$.

\subsubsection{\texorpdfstring{The spectrum for large $\lambda$}{The spectrum for large lambda}}
We begin by looking at $\negeig(\+{\oper{M}}_{\varepsilon,n}^{\lambda})$ when $\lambda$ is large. This turns out to be relatively simple due to the block form of the untruncated operator.

\begin{lemma}\label{finding lambda upper *}
There is $\lambda^*>0$ such that for all $\lambda\ge\lambda^*$, $\varepsilon>0$, and $n\in\mathbb{N}$, the truncated operator $\M_{\varepsilon,n}$ has spectrum composed of exactly $n$ positive and $n$ negative eigenvalues. In particular $\negeig(\+{\oper{M}}^\lambda_{\varepsilon,n})=n$.
\end{lemma}

\begin{proof}
Take $\+u=(u_1,0)\in L^2(\mathbb{R})\times L^2_0(\Omega)$ with $u_1\in\dom{\oper{A}^\lambda_{2,\varepsilon}}$, $\norm{\+u}_{L^2(\mathbb{R})\times L^2(\Omega)}=1$ and $\+u$ in the $2n$ dimensional subspace associated with the truncation. Then,
\begin{equation*}
\begin{aligned}
\ip{\+{\oper{M}}^{\lambda}_{\varepsilon,n}u}{u}_{L^2(\mathbb{R})\times L^2_0(\Omega)}&=\ip{\oper{A}_{1,\varepsilon,n}^{\lambda}u_1}{u_1}_{L^2(\mathbb{R})}=\ip{\oper{A}_{1,\varepsilon}^{\lambda}u_1}{u_1}_{L^2(\mathbb{R})}\\
&=\ip{\oper{A}_{1}^{\lambda}u_1}{u_1}_{L^2(\mathbb{R})}+\varepsilon\norm{\sqrt{W}u_1}^2_{L^2(\mathbb{R})}.
\end{aligned}
\end{equation*}
As the second term is non-negative we may apply \autoref{lem:prop A}(d) to see that, for all large enough $\lambda$ (independently of $n$ and $\varepsilon$), $\+{\oper{M}}^{\lambda}_{\varepsilon,n}$ is positive definite on a subspace of dimension $n$, so has $n$ positive eigenvalues. Performing the same computation on $\+u=(0,u_2)$ in the subspace associated with the truncation and with $u_2\in\dom{\oper{A}_{1,\varepsilon}^\lambda}$, we obtain that for large enough $\lambda$, $\+{\oper{M}}^{\lambda}_{\varepsilon,n}$ is negative definite on a subspace of dimension $n$. As $\+{\oper{M}}^{\lambda}_{\varepsilon,n}$ has exactly $2n$ eigenvalues the proof is complete.
\end{proof}

\subsubsection{\texorpdfstring{The spectrum for small $\lambda$}{The spectrum for small lambda}}
We now consider $\mathrm{sp}(\+{\oper{M}}^{\lambda_*}_{\varepsilon,n})$. We recall the result on spectra of real block matrix operators in \cite{Ben-Artzi2011b}:
\begin{lemma}\label{eigenvalue counts of matrices}
Let $M$ be the real symmetric block matrix
\begin{equation*}
M=\begin{bmatrix}
   A_2&B^T\\
   B&-A_1
  \end{bmatrix}
\end{equation*}
with $A_1$ invertible. Then $M$ has the same number of negative eigenvalues as the matrix
\begin{equation*}
N=\begin{bmatrix}
   A_2+B^TA_1^{-1}B&0\\
   0&-A_1
  \end{bmatrix}.
\end{equation*}
\end{lemma}
%

\begin{lemma}\label{really finding lambda lower * 15d}
Assume  that \eqref{eigcountcond1} holds and that zero is in the resolvent set of $\oper{A}_1^0$. Then there exist $\lambda_*,\varepsilon_*>0$ such that for all $\varepsilon\in(0,\varepsilon_*)$ there is $N>0$ such that for all $n>N$ the operator $\+{\oper{M}}^{\lambda_*}_{\varepsilon,n}$ satisfies
\begin{equation*} 
 \negeig(\+{\oper{M}}^{\lambda_*}_{\varepsilon,n})\ge \negeig(\oper{A}^{0}_2+(\oper{B}^0)^*(\oper{A}_1^{0})^{-1}\oper{B}^{0})+n-\negeig(\oper{A}_1^0).
\end{equation*}
\end{lemma}
\begin{proof}
 The number $\lambda_*$ is the one given in \autoref{eigcountcond1movedlemma}, and satisfies that for all $\lambda\in[0,\lambda_*]$ the kernel of $\oper{A}_1^\lambda$ is trivial. Since eigenvalues (counting multiplicity) are stable under strong resolvent perturbations (see \cite[VIII.3.5.Thm 3.15.]{Kato1995}), there exists $\varepsilon_*>0$ such that $\negeig(\oper{A}^{0}_{2,\varepsilon}+(\oper{B}^0)^*(\oper{A}_1^{0})^{-1}\oper{B}^{0})\geq\negeig(\oper{A}^{0}_2+(\oper{B}^0)^*(\oper{A}_1^{0})^{-1}\oper{B}^{0})$ for all $\varepsilon\in[0,\varepsilon_*]$. The result then follows from \autoref{eigenvalue counts of matrices}, since $\negeig(\+{\oper{M}}^{\lambda_*}_{\varepsilon,n})= \negeig(\oper{A}^{\lambda_*}_{2,\varepsilon}+(\oper{B}^{\lambda_*})^*(\oper{A}_1^{\lambda_*})^{-1}\oper{B}^{\lambda_*})+n-\negeig(\oper{A}^{\lambda_*}_1)$.
\end{proof}

\subsection{The cylindrically symmetric case}\label{sec:finding kernel 3d}
For brevity we write
	\begin{equation*}
	\Mt
	=
	\begin{bmatrix}
	\widetilde{\oper{A}}_2^\lambda & (\widetilde{\+{\oper{B}}}_4^\lambda)^*\\
	\widetilde{\+{\oper{B}}}_4^\lambda & -\widetilde{\+{\oper{A}}}_4^\lambda
	\end{bmatrix}
	\end{equation*}
where
	\begin{equation*}
	\widetilde{\+{\oper{A}}}_4^\lambda
	=
	\begin{bmatrix}
	\widetilde{\oper{A}}_1^\lambda&(\widetilde{\oper{B}}_3^\lambda)^*\\
	\widetilde{\oper{B}}_3^\lambda&\widetilde{\oper{A}}_3^\lambda	
	\end{bmatrix}\qquad\text{and}\qquad\widetilde{\+{\oper{B}}}_4^\lambda
	=
	\begin{bmatrix}
	\widetilde{\oper{B}}_1^\lambda \\
	\widetilde{\oper{B}}_2^\lambda
	\end{bmatrix}.
	\end{equation*}
\subsubsection{Continuity of the spectrum at \texorpdfstring{$\lambda=0$}{lambda=0}}\label{subsec:Moving instability criterion 3d}
\begin{lemma}
Assume that \eqref{eq:thm2-condition} holds, that $\widetilde{\oper{A}}_3^0$ does not have $0$ as an $L^6$-eigenvalue (see \autoref{def:L6kernel}) and that $\widetilde{\oper{A}}_1^0$ does not have $0$ as an eigenvalue.  Then there exists $\lambda_*>0$ such that for $\lambda\in[0,\lambda_*]$,
\begin{equation*}
\negeig(\widetilde{\oper{A}}_2^\lambda+(\widetilde{\+{\oper{B}}}_4^\lambda)^*(\widetilde{\+{\oper{A}}}_4^\lambda)^{-1}\widetilde{\+{\oper{B}}}_4^\lambda)> \negeig(\widetilde{\+{\oper{A}}}_4^\lambda).
\end{equation*}
\end{lemma}
\begin{proof}
We first note that as the mean perturbed charge is zero, it follows from direct computation on the Green's function of the Laplacian (see \cite[Lemma 3.2]{Lin2007}) that any $L^6$-eigenfunction of $\widetilde{\oper{A}}_1^0$ will also be square integrable and so be a proper eigenfunction.
Note also that $\widetilde{\oper{B}}_3^0=0$. Thus our assumptions imply that $\widetilde{\+{\oper{A}}}_4^0$ has no $L^6$-eigenfunction of $0$.
\newline

We model the proof on that of \autoref{eigcountcond1movedlemma}, splitting it into 4 steps.

\paragraph{Step 1. $\widetilde{\+{\oper{A}}}_4^\lambda$ is invertible for small $\lambda\ge0$ when restricted to functions supported in $\Omega$. }
Let $\+{\oper{P}}\in\bounded{L^2_{cyl}(\mathbb{R}^3)\times L^2_{rz}(\mathbb{R}^3;\mathbb{R}^3)}$ be multiplication by the indicator function of $\Omega$. We claim that for all small enough $\lambda>0$, $\+{\oper{P}}(\widetilde{\+{\oper{A}}}_4^\lambda)^{-1}\+{\oper{P}}$ is a well defined bounded operator that is strongly continuous in $\lambda>0$ and has a strong limit as $\lambda\to0$. To prove this, we argue that if this were not the case, then zero would be an $L^6$-eigenvalue, a contradiction.

As $L^2_{cyl}(\mathbb{R}^3)\times L^2_{rz}(\mathbb{R}^3;\mathbb{R}^3)$ is a closed subspace of $L^2(\mathbb{R}^3;\mathbb{R}^4)$ we may work in the larger space to ease notation. To this end, let $\norm{\cdot}$ and $\ip{\cdot}{\cdot}$ denote the $L^2(\mathbb{R}^3;\mathbb{R}^4)$ norm and inner product. We can express $\widetilde{\+{\oper{A}}}_4^\lambda$ in the form
\begin{equation*}
\widetilde{\+{\oper{A}}}_4^\lambda\+u=-\+\Delta \+u+\lambda^2\+u+\+{\oper{K}}^\lambda\+u
\end{equation*}
where $\+{\oper{K}}^\lambda$ is uniformly bounded, strongly continuous in $\lambda\ge0$ and $\+{\oper{K}}^\lambda=\+{\oper{P}}\+{\oper{K}}^\lambda\+{\oper{P}}$.

\subparagraph{Step 1.1. $\widetilde{\+{\oper{A}}}_4^\lambda$ is bounded below when restricted to functions supported in $\Omega$.}
First we claim that there exist constants $\lambda'>0$ and $C>0$ such that we have the uniform lower bound
\begin{equation}\label{operator lower bound}
 \norm{\mathbbm{1}_\Omega\widetilde{\+{\oper{A}}}_4^\lambda\+u^\lambda}\ge C\norm{\mathbbm{1}_\Omega \+u^\lambda},\qquad \forall\lambda\in(0,\lambda']
\end{equation}
where the constant $C$ does not depend on $\lambda$ or on $\+u^\lambda$ and where $\+u^\lambda$ satisfies $\widetilde{\+{\oper{A}}}_4^\lambda\+ u^\lambda=0$ outside $\Omega$.
Indeed, if not there would be sequences $\lambda_n\to0$, $\+u_n$ with $\norm{\mathbbm{1}_\Omega\+u_n}_{L^2}=1$ that satisfy
\begin{equation}\label{equation for operator lower bound1}
\widetilde{\+{\oper{A}}}_4^{\lambda_n}\+u_n=-\+\Delta \+u+\lambda^2_n\+u_n+\+{\oper{K}}^{\lambda_n}\+u_n=\+f_n\to0
\end{equation}
as $n\to\infty$, with $\+f_n$ supported in $\Omega$. Hence,
\begin{equation*}\label{equation for operator lower bound2}
\lambda_n^2\norm{\+u_n}^2+\norm{\nabla \+u_n}^2+\ip{\+{\oper{K}}^{\lambda_n}\+u_n}{\+u_n}=\ip{\+f_n}{\+u_n}\to0
\end{equation*}
so that $\norm{\nabla \+u_n}^2$ is uniformly bounded for large enough $n$. Hence there exists a subsequence (we abuse notation and keep the same sequence) such that $ \nabla \+u_n\tow \+v\text{ weakly in }L^2(\mathbb{R}^3;\mathbb{R}^4)$ for some $\+v\in L^2(\mathbb{R}^3;\mathbb{R}^4)$. By the standard Sobolev inequality $\norm{\varphi}_{L^6(\mathbb{R}^3)}\le C\norm{\nabla\varphi}_{L^2(\mathbb{R}^3)}$ we have $ \+u_n\tow \+u\text{ weakly in }L^6(\mathbb{R}^3;\mathbb{R}^4)$ for some $ \+u\in L^6(\mathbb{R}^3;\mathbb{R}^4)$. Furthermore, by Rellich's theorem we have $ \+u_n\to \+u\text{ in }L^2_{loc}(\mathbb{R}^3;\mathbb{R}^4)$. This implies that necessarily $\+v=\nabla\+u$. In particular we deduce that the limit $\+u$ is in $L^6(\mathbb{R}^3;\mathbb{R}^4)$ and $\norm{\mathbbm{1}_\Omega\+u}=1$ so $\+u\ne0$. Passing to the limit in \eqref{equation for operator lower bound1}, $\+u$ satisfies
\begin{equation*}
-\+\Delta \+u+\+{\oper{K}}^0\+u=0
\end{equation*}
in the sense of distributions and by elliptic regularity $\+u\in H^2_{loc}(\mathbb{R}^3;\mathbb{R}^4)$. In fact $u$ is an $L^6$-eigenfunction of $0$, which we assumed impossible. This proves the claim.

\subparagraph{Step 1.2. $\widetilde{\+{\oper{A}}}_4^\lambda$ is invertible for all small enough $\lambda>0$.}

For any $\lambda>0$, $0$ does not lie in the essential spectrum of $\widetilde{\+{\oper{A}}}_4^\lambda$ so is either an eigenvalue or in the resolvent set. Let $\lambda>0$ be small enough that \eqref{operator lower bound} holds, then any eigenfunction $\+u$ of $0$ satisfies all the assumptions of the claim above, and hence $\norm{\mathbbm{1}_\Omega\+u}\le C \norm{\mathbbm{1}_\Omega\widetilde{\+{\oper{A}}}_4^\lambda\+u}=0$ so that $\+u=0$ inside $\Omega$. Clearly this implies that $\+u=0$ in $\mathbb{R}^3$ which is a contradiction. In the same way we deduce a uniform bound $1/C$ for the operator $\+{\oper{P}}(\widetilde{\+{\oper{A}}}_4^\lambda)^{-1}\+{\oper{P}}$ for such small $\lambda>0$.

\subparagraph{Step 1.3. $\+{\oper{P}}(\widetilde{\+{\oper{A}}}_4^0)^{-1}\+{\oper{P}}$ is well defined and bounded.}
Finally we give a meaning to $\+{\oper{P}}(\widetilde{\+{\oper{A}}}_4^0)^{-1}\+{\oper{P}}$, (which is required as $\widetilde{\+{\oper{A}}}_4^0$ is not invertible on the whole space). We define it to be the strong operator limit of $\+{\oper{P}}(\widetilde{\+{\oper{A}}}_4^\lambda)^{-1}\+{\oper{P}}$ as $\lambda\to0$. Indeed, suppose that $\+f$ is fixed with support in $\Omega$ and $\lambda_n\to0$. Then we wish to compute the limit of $\+{\oper{P}}\+u_n$ for $\+u_n=(\widetilde{\+{\oper{A}}}_4^{\lambda_n})^{-1}\+{\oper{P}}\+f$ as $n\to\infty$ and show that it is independent of the sequence $\lambda_n\to0$. Indeed $\+u_n$ will satisfy
\begin{equation*}
\widetilde{\+{\oper{A}}}_4^{\lambda_n}\+u_n=\lambda^2_n\+u_n-\+\Delta \+u_n+\+{\oper{K}}^{\lambda_n}\+u_n=\+f.
\end{equation*}
By the same argument as before we can extract a subsequence and limit $\+u\in L^6(\mathbb{R}^3;\mathbb{R}^4)$ with convergences as in Step 1.1. In particular $\+{\oper{P}}\+u_n\to \+{\oper{P}}\+ u$. Furthermore, any convergent subsequence $\+u_{n_k}$ for any sequence $\lambda_n\to0$ has the same limit of $\+{\oper{P}}\+u_{n_k}$ as otherwise their difference would be an $L^6$-eigenfunction in the kernel, which we assumed to not exist. We have shown that for any sequences $\lambda_n\to0$ and $\+u_n$ we can extract a subsequence converging to a limit that does not depend on the subsequence or the original sequence, therefore the whole sequence converges to this limit, which completes the proof of the existence of the strong operator limit. The bound in operator norm is carried over by strong continuity.

\paragraph{Step 2. $\negeig(\widetilde{\+{\oper{A}}}_4^\lambda)=\negeig(\widetilde{\+{\oper{A}}}_4^0)$ for all $\lambda\in[0,\lambda_*]$.} $\widetilde{\+{\oper{A}}}_4^\lambda$ is norm resolvent continuous in $\lambda\ge0$, so the only way the number of negative eigenvalues could change is for an eigenvalue to be absorbed into the essential spectrum at $0$ as $\lambda\to0$. Assume this happens, then we have a sequence $\lambda_n\to0$, a sequence of negative eigenvalues $\sigma_n\to0$ and eigenfunctions $\+u_n$ which satisfy
\begin{equation*}
-\+\Delta\+u_n+\lambda_n^2\+u_n+\+{\oper{K}}^{\lambda_n}\+u_n=\sigma_n\+u_n.
\end{equation*}
By the same argument as in the previous steps, we may take subsequences and obtain a contradiction.

\paragraph{Step 3. $\negeig(\widetilde{\oper{A}}_2^\lambda+(\widetilde{\+{\oper{B}}}_4^\lambda)^*(\widetilde{\+{\oper{A}}}_4^\lambda)^{-1}\widetilde{\+{\oper{B}}}_4^\lambda)\ge \negeig(\widetilde{\oper{A}}_2^0+(\widetilde{\+{\oper{B}}}_4^0)^*(\widetilde{\+{\oper{A}}}_4^0)^{-1}\widetilde{\+{\oper{B}}}_4^0)$ for all $\lambda\in[0,\lambda_*]$.}
This may be proved in the same way as Step 3 of \autoref{eigcountcond1movedlemma}.
\paragraph{Step 4. $\negeig(\widetilde{\oper{A}}_2^0+(\widetilde{\+{\oper{B}}}_4^0)^*(\widetilde{\+{\oper{A}}}_4^0)^{-1}\widetilde{\+{\oper{B}}}_4^0)>\negeig(\widetilde{\+{\oper{A}}}_4^0)$.}
As $\widetilde{\oper{B}}_2^0=0$ and $\widetilde{\oper{B}}_3^0=0$ we have,
\begin{equation*}
\begin{aligned}
\negeig&(\widetilde{\oper{A}}_2^0+(\widetilde{\+{\oper{B}}}_4^0)^*(\widetilde{\+{\oper{A}}}_4^0)^{-1}\widetilde{\+{\oper{B}}}_4^0)\\
&=\negeig(\widetilde{\oper{A}}_2^0+(\widetilde{\oper{B}}_1^0)^*(\widetilde{\oper{A}}_1^0)^{-1}\widetilde{\oper{B}}^0_1)>\negeig(\widetilde{\oper{A}}_1^0)+\negeig(\widetilde{\oper{A}}_3^0)=\negeig(\widetilde{\+{\oper{A}}}_4^0)
\end{aligned}
\end{equation*}
where the inequality is obtained from the assumption of the lemma.
\end{proof}

\subsubsection{Finding a non-trivial kernel}
The next few steps of the proof follow those of the $1.5d$ case, hence we only provide a short overview.

\emph{Truncation.}
As the domain is unbounded, each Laplacian appearing in the problem contributes an essential spectrum on $[0,\infty)$. We therefore introduce a smooth positive potential function $W:\mathbb{R}^3\to\mathbb{R}$ satisfying $W(x)\to\infty$ as $|x|\to\infty$ and denote by $W^{\otimes n}$ the $n$-dimensional vector-valued function with $n$ copies of $W$. Then we define
\begin{equation*}
\widetilde{\+{\oper{M}}}^\lambda_\varepsilon
=
	\begin{bmatrix}
	\widetilde{\oper{A}}_{2,\varepsilon}^\lambda & (\widetilde{\+{\oper{B}}}_4^\lambda)^*\\
	\widetilde{\+{\oper{B}}}_4^\lambda & -\widetilde{\+{\oper{A}}}_{4,\varepsilon}^\lambda
	\end{bmatrix}
=
\underbrace{\begin{bmatrix}
-\Delta+\varepsilon W^{\otimes 3}&0\\
0&\Delta-\varepsilon W^{\otimes 4}
\end{bmatrix}
+
\begin{bmatrix}
\lambda^2&0\\
0&-\lambda^2
\end{bmatrix}}_{\widetilde{\+{\oper{A}}}_\varepsilon^\lambda}
-\widetilde{\+{\oper{J}}}^\lambda.
\end{equation*}
As above we can naturally define finite-dimensional operators $\widetilde{\+{\oper{M}}}_{\varepsilon,n}^{\lambda}$, for which we can easily prove:
\begin{lemma}\label{spectral continuity argument 3d}
Fix $\varepsilon>0,n\in\mathbb{N}$. Suppose that there exist $0<\lambda_*<\lambda^*<\infty$ such that $\negeig({\widetilde{\+{\oper{M}}}}_{\varepsilon,n}^{\lambda_*})\ne\negeig(\widetilde{\+{\oper{M}}}_{\varepsilon,n}^{\lambda^*})$. Then there is a $\lambda_{\varepsilon,n}\in(\lambda_*,\lambda^*)$ for which $\ker({\widetilde{\+{\oper{M}}}}^\lambda_{\varepsilon,n})$ is non-trivial.
\end{lemma}

\emph{The spectrum for large $\lambda$.}
This is again similar to the $1.5d$ case, in particular due to the appearance of the $\lambda^2$ terms. We have:
\begin{lemma}\label{finding lambda upper * 3d}
There is $\lambda^*>0$ such that for all $\lambda\ge\lambda^*$, $\varepsilon>0$, and $n\in\mathbb{N}$, the truncated operator $\Mt_{\varepsilon,n}$ has spectrum composed of exactly $n$ positive and $n$ negative eigenvalues. In particular $\negeig({\widetilde{\+{\oper{M}}}}^\lambda_{\varepsilon,n})=n$.
\end{lemma}

\emph{The spectrum for small $\lambda$.}
Again this is similar to the $1.5d$ case.
\begin{lemma}\label{really finding lambda lower *}
Assume  that \eqref{eq:thm2-condition} holds and that zero is neither an eigenvalue of $\widetilde{\oper{A}}_1^0$ nor is it an $L^6$-eigenvalue of $\widetilde{\oper{A}}_3^0$. Then there exist $\lambda_*,\varepsilon_*>0$ such that for all $\varepsilon\in(0,\varepsilon_*)$ there is $N>0$ such that for all $n>N$ the operator $\+{\widetilde{\oper{M}}}^{\lambda_*}_{\varepsilon,n}$ satisfies
\begin{equation*}
 \negeig({\widetilde{\+{\oper{M}}}}^{\lambda_*}_{\varepsilon,n})\ge \operatorname{neg}\left(\widetilde{\oper{A}}_2^0+\left(\widetilde{\oper{B}}_1^0\right)^*\left(\widetilde{\oper{A}}_1^0\right)^{-1}\widetilde{\oper{B}}_1^0\right)+n-\negeig\left({\widetilde{\oper{A}}}_1^0\right)-\negeig\left({\widetilde{\oper{A}}}_3^0\right).
\end{equation*}
\end{lemma}

\section{Proofs of the main theorems}\label{sec:existence}
In this section we complete the proofs of \autoref{theorem1} and \autoref{theorem2}. In both settings -- the $1.5d$ and the cylindrically symmetric -- we first show that the results of \autoref{sec:solving} imply that there exists some $\lambda>0$ such that the equivalent problems \eqref{Formally definition of M} and \eqref{Formally definition of M 2} have a non-trivial solution (the $\lambda$ need not be the same in both cases, of course). Then we show that these non-trivial solutions lead to genuine non-trivial solutions of the linearised RVM in either case.

\subsection{The \texorpdfstring{$1.5d$}{1.5d} case}\label{sec:existence 15d}
\subsubsection{Existence of a non-trivial kernel of the equivalent problem}\label{sec:exist non triv ker 1.5d}
By \autoref{finding lambda upper *} and \autoref{really finding lambda lower * 15d} we have $0<\lambda_*<\lambda^*<\infty$ and $\varepsilon_*>0$ such that for any $\varepsilon<\varepsilon_*$ there is an $N_{\varepsilon}$ such that for $n>N_{\varepsilon}$ we have,
\begin{equation*}
\begin{aligned}
\negeig(\oper{M}^{\lambda_*}_{\varepsilon,n})&\ge \negeig(\oper{A}^{0}_2+(\oper{B}^0)^*(\oper{A}_1^{0})^{-1}\oper{B}^{0})+n-\negeig(\oper{A}_1^0)\\
&>n=\negeig(\+{\oper{M}}^{\lambda^*}_{\varepsilon,n}),
\end{aligned}
\end{equation*}
where the strict inequality is due to the assumption \eqref{eq:thm1-condition}.  Fix $\varepsilon\in(0,\varepsilon_*)$. By \autoref{spectral continuity argument} for each $n>N_{\varepsilon}$ there exists $\lambda_{\varepsilon,n}\in(\lambda_*,\lambda^*)$ such that $0\in\operatorname{sp}(\+{\oper{M}}^{\lambda_{\varepsilon,n}}_{\varepsilon,n})$. By compactness of the interval $[\lambda_*,\lambda^*]$ we may pass to a subsequence where $\lambda_{\varepsilon,n_k}\to\lambda_{\varepsilon}$ as $k\to\infty$, for some $\lambda_{\varepsilon}\in[\lambda_*,\lambda^*]$. By \autoref{thm:approximation} the spectra of the approximations $\+{\oper{M}}^{\lambda_{\varepsilon,n_k}}_{\varepsilon,n_k}$ converge in Hausdorff distance in $(-K,\lambda_*^2)$ to the spectrum of $\+{\oper{M}}^{\lambda_\varepsilon}_{\varepsilon}$, where $K>0$ is the spectral gap of the Neumann Laplacian $\partial_x^2$ on $L^2_0(\Omega)$. This implies that $0\in \operatorname{sp}(\+{\oper{M}}^{\lambda_{\varepsilon}}_{\varepsilon})$. We now repeat this argument to send $\varepsilon\downarrow0$, obtaining $\lambda\in[\lambda_*,\lambda^*]$ with $0\in\operatorname{sp}(\+{\oper{M}}^{\lambda})$. Finally, the discreteness of the spectrum of $\+{\oper{M}}^{\lambda}$ in $(-\infty,\lambda^2)$, (\autoref{lem:prop M}), ensures that $0$ is an eigenvalue of $\+{\oper{M}}^{\lambda}$, i.e. $\+{\oper{M}}^{\lambda}$ has a non-trivial kernel.

\subsubsection{Existence of a growing mode}
Now that we know that there exist some $\lambda\in(0,\infty)$ and some $u=\begin{bmatrix}\psi&\phi\end{bmatrix}^T\in H^2(\mathbb{R})\times H^2_{0,n}(\Omega)$ that solve \eqref{Formally definition of M} we show that a genuine growing mode as defined in \eqref{eq:purely-growing} really exists. To this end, we use $\phi,\psi$ and $\lambda$ to define
	\[
	E_1=-\partial_x\phi \qquad E_2=-\lambda\psi \qquad B=\partial_x\psi
	\]
(which lie in $H^1(\Omega)$, $H^2(\mathbb{R})$ and $H^1(\mathbb{R})$, respectively) and to define $f^\pm(x,v)$ as in \eqref{f in 1.5d}:
	\[f^\pm=\pm\mu^\pm_e\phi\pm{\mu_p^\pm}\psi\pm\lambda(\lambda+\D)^{-1}[\mu_e^\pm(-\phi+\hat{v}_2\psi)].\]
Observe that $f^\pm$ are both in $L^2(\mathbb{R}\times\mathbb{R}^2)$ since $\mu_e^\pm$ and $\mu_p^\pm$ are continuous functions that are compactly supported in the spatial variable which satisfy the integrability condition \eqref{eq:int-condition}. In fact, $f^\pm$ are in the domains of $\D$, respectively, since $e^\pm$ and $p^\pm$ are constant along trajectories and $\phi$ and $\psi$ are twice differentiable.

\begin{lemma}\label{cont eq 1.5d}
The functions $f^\pm$ solve the linearised Vlasov equations \eqref{final linearised RVM}.
\end{lemma}
\begin{proof}
This is almost a tautology: applying the operators $\lambda+\D$ to the expressions for $f^\pm$, respectively, one is left precisely with the expressions \eqref{final linearised RVM}.
\end{proof}


\begin{lemma}
The functions $f^\pm$ belong to $L^1(\mathbb{R}\times\mathbb{R}^2)$.
\end{lemma}
\begin{proof}
Dropping the $\pm$ for brevity, the first term making up $f$ is estimated as follows	
	\begin{equation*}
	\|\mu_e\phi\|_{L^1(\mathbb{R}^3)}\lesssim\|\mu_e\|_{L^2(\mathbb{R}^3)}\|\phi\|_{L^2(\mathbb{R})}\lesssim\|\mu_e\|^{1/2}_{L^\infty(\mathbb{R}^3)}\|\mu_e\|^{1/2}_{L^1(\mathbb{R}^3)}\|\phi\|_{L^2(\mathbb{R})}<\infty.
	\end{equation*}
The other terms are estimated similarly (for the terms involving the averaging operator this may be seen by writing the ergodic average explicitly (see \autoref{rem:previous-definitions-of-Q}) or by using boundedness of the averaging operator on $\Ls_\pm$). This implies that $f^\pm\in L^1(\mathbb{R}\times\mathbb{R}^2)$.
\end{proof}

We now define the charge and current densities $\rho$ and $j_i$ by
 \begin{equation*}
	\rho=\int (f^+-f^-)\;dv \qquad j_i=\int \hat{v}_i (f^+-f^-)\;dv,\quad i=1,2.
 \end{equation*}
 Integrating $f^\pm$ in the momentum variable $v$ alone, we obtain that $\rho\in L^1(\mathbb{R})$ as well as $j_i\in L^1(\mathbb{R})$ since $|\hat{v}_i|\leq1$. In particular $\rho,j_i$ are distributions on $\mathbb{R}$.

\begin{lemma}
The continuity equation $\lambda\rho+\partial_xj_1=0$ holds in the sense of distributions.
\end{lemma}

\begin{proof}
This follows from integrating the linearised Vlasov equations in the momentum variable. Indeed, we informally have
	\begin{align*}
	 \int(\lambda+\D)f^\pm\,dv&=\pm\int\left[(\lambda+\D)(\mu^\pm_e\phi+{\mu_p^\pm}\psi)+\lambda\mu^\pm_e(-\phi+\hat{v}_2\psi)\right]\,dv\\
	 &=\pm\int\lambda\psi\left(\mu_p^\pm+\mu_e^\pm\hat{v}_2\right)\,dv\pm\int\D\left(\mu^\pm_e\phi+{\mu_p^\pm}\psi\right)\,dv=0,
	\end{align*}
where the first term on the right hand side vanishes due to the identity \eqref{eq:perfect deriv} and the second term vanishes since $\mu^\pm$ are even in $\hat{v}_1$, whereas $\D=\hat{v}_1\partial_x$ when applied to functions of $x$ alone (recall that $\mu^\pm$ are constant along trajectories of $\D$, as are $\mu^\pm_e$ and $\mu^\pm_p$). We obtain the continuity equation by subtracting the ``$-$'' expression above from the ``$+$'' expression. Owing to the low regularity of $f^\pm$, $\rho$ and $j_1$ this is true in a weak sense.
\end{proof}

\begin{lemma}\label{Recovery of Maxwell's equations 1.5d}
Maxwell's equations  \eqref{Maxwell's equations 1.5d} hold.
\end{lemma}

\begin{proof}
Equations \eqref{Ampere's equation 1 1.5d} and \eqref{Gauss's equation 1.5d} hold due to \eqref{Formally definition of M} and the definitions of the operators \eqref{eq:operators}. Indeed, from the second line of \eqref{Formally definition of M}, we have
\begin{align*}
	0
	&=
	\left(\int\sum_\pm\mu_p^\pm\,d\+v\right) \psi+\int\sum_\pm\mu_e^\pm\Q~[\hat{v}_2\psi]\,d\+v+\partial_x^2 \phi-\int\sum_{\pm}\mu_e^\pm(\Q-1) \phi\,d\+v\\
	&=\partial_x^2 \phi+\int\sum_\pm\left(\mu_p^\pm\psi+\mu_e^\pm\Q~[\hat{v}_2 \psi]-\mu_e^\pm(\Q-1) \phi\right)\, d\+v\\
	&\stackrel{\eqref{f in 1.5d}}{=}
	\partial_x^2\phi+\int (f^+-f^-)\,d\+v.
	\end{align*}
	which is  \eqref{Gauss's equation 1.5d}. Similarly, \eqref{Ampere's equation 1 1.5d} is obtained from the first line of \eqref{Formally definition of M}.
	
 We therefore just need to show that \eqref{Ampere's equation 2 1.5d} holds. However this is a simple consequence of \eqref{Gauss's equation 1.5d} and the continuity equation. Indeed, we may first write
	\begin{equation*}
	-\lambda \partial_xE_1
	=
	\lambda \partial^2_x\phi
	\stackrel{\eqref{Gauss's equation 1.5d}}{=}
	-\lambda\rho
	\stackrel{\text{cont. eq.}}{=}
	\partial_xj_1
	\end{equation*}
which is the derivative of \eqref{Ampere's equation 2 1.5d}. Next, as $\phi\in H^2_{n,0}(\Omega)$, its derivative $E_1$ vanishes on $\partial\Omega$, and $j_1$ also vanishes there due to the compact support of the equilibrium in $\Omega$. Thus, $-\lambda E_1$ and $j_1$ have the same derivative inside $\Omega$ and the same values on $\partial\Omega$, which means they must be equal.
\end{proof}

\begin{lemma}
 The charge and current densities $\rho$, $j_1$ and $j_2$ are elements in $L^1(\mathbb{R})\cap L^2(\mathbb{R})$.
\end{lemma}

\begin{proof}
This follows from Maxwell's equations and the regularity of $\psi,\phi$ which are in $H^2(\mathbb{R})$ and $H^2_{0,n}(\Omega)$ respectively.
 %
%
\end{proof}

%
%
%
%
%
%
%
%
%

This concludes the proof of \autoref{theorem1}.

\subsection{The cylindrically symmetric case}
\subsubsection{Existence of a non-trivial kernel of the equivalent problem}
The proof of the existence of a non-trivial kernel in the cylindrically symmetric case is in complete analogy to the one in the $1.5d$ case presented in \autoref{sec:exist non triv ker 1.5d} and is therefore omitted.

\subsubsection{Existence of a growing mode}

Let $\lambda>0$ and $u=\begin{bmatrix}\?A_\theta&\varphi&\?A_{rz}\end{bmatrix}^T\in H^2_\theta(\mathbb{R}^3;\mathbb{R}^3)\times H^2_{cyl}(\mathbb{R}^3)\times H^2_{rz}(\mathbb{R}^3;\mathbb{R}^3)$ be such that  \eqref{Formally definition of M 2} is satisfied, i.e. $\Mt u=0$. Let $H^2_{cyl}(\mathbb{R}^3;\mathbb{R}^3)\ni\?A=\?A_\theta+\?A_{rz}$ as in \eqref{eq:A-decompose} and define
	\[
	\?E=-\nabla\varphi \qquad  \qquad \?B=\nabla\times\?A
	\]
(which each lie in $H^1_{cyl}(\mathbb{R}^3;\mathbb{R}^3)\subseteq H^1(\mathbb{R}^3;\mathbb{R}^3)$). Furthermore, define
	\[f^\pm=\pm\mu_e^\pm\varphi\pm r\mu_p^\pm(\?A\cdot\+e_\theta)\pm\mu_e^\pm\lambda(\lambda+\Dt)^{-1}(-\varphi+\?A\cdot\+{\hat{v}}).\]

As in the $1.5d$ case we begin by establishing that $f^\pm$ are integrable and satisfy the linearised Vlasov and continuity equations. The proof of this result is analogous to the corresponding results in the $1.5d$ case, so is omitted.
\begin{lemma}\label{cont eq 3d}
The functions $f^\pm$ solve the linearised Vlasov equations \eqref{RVM 3D linearised1} in the sense of distributions, and belong to $L^1(\mathbb{R}^3\times\mathbb{R}^3)$. Furthermore, the charge and current densities $\rho$ and $\?j$ defined by
\begin{equation*}
	\rho=\int (f^+-f^-)\;d\+v \qquad \?j=\int \+{\hat{v}} (f^+-f^-)\;d\+v,
 \end{equation*}
 belong to $L^1(\mathbb{R}^3)$ and $L^1(\mathbb{R}^3;\mathbb{R}^3)$, respectively, and satisfy the continuity equation $\lambda\rho+\nabla\cdot\?j=0$ in the sense of distributions.
\end{lemma}
Next we recover Maxwell's equations from \eqref{Formally definition of M 2} and the continuity equation.

\begin{lemma}\label{Elimination of Ampere 2 3d}
Both the Lorenz gauge condition $\lambda\varphi+\nabla\cdot\?A=0$ (see \eqref{eq:lin-lorenz-gauge})   and Maxwell's equations \eqref{eq:maxwell-lorenz} are satisfied.
\end{lemma}
\begin{proof}
In the same way as the 1.5d case, \eqref{eq:maxwell-lorenz-charge} is obtained from the second line of \eqref{Formally definition of M 2},\eqref{eq:m lambda 2 3d} and the definitions, \eqref{eq:operators3d}, of the operators therein. Similarly, \eqref{eq:maxwell-lorenz-current} is obtained from the first and third lines of \eqref{Formally definition of M 2},\eqref{eq:m lambda 2 3d}.

It remains to show that the Lorenz gauge condition holds. Using \eqref{eq:maxwell-lorenz-charge} and \eqref{eq:maxwell-lorenz-current} in the continuity equation, we have, in the sense of distributions,
\begin{equation*}
\begin{aligned}
0&=\lambda(-\Delta+\lambda^2)\varphi+\nabla\cdot[(-\+\Delta+\lambda^2)\?A]\\
&=(-\Delta+\lambda^2)[\lambda \varphi+\nabla\cdot\?A].
\end{aligned}
\end{equation*}
As $-\Delta+\lambda^2$ is invertible, this implies that $\lambda \varphi+\nabla\cdot\?A=0$.
\end{proof}

This concludes the proof of \autoref{theorem2}.

\section{Properties of the operators}\label{sec:properties}
Here we gather all important properties of the operators defined in \autoref{sec:equiv}, as well as the operators defined in \eqref{the 0 operators 1.5d} and  \eqref{eq:the 0 operators 3d}.

\subsection{The \texorpdfstring{$1.5d$}{1.5d} case}\label{sec:prop operators 1.5d}
As the only dependence on $\lambda$ is through the operators $\Q$ we start with them:

\begin{lemma}\label{properties of Q}
In the respective spaces $\Ls_\pm$, $\Q$ satisfy:
\begin{enumerate}[(a)]
 \item $\norm{\Q}_{\bounded{\Ls_\pm}}=1$.
\item $\Q$ can be extended from $\lambda>0$ to $\operatorname{Re}\lambda>0$ as a holomorphic operator valued function. In particular it is continuous for $\lambda>0$ in operator norm topology.
\item As $\mathbb{R}\ni\lambda\to\infty$, $\Q\tos1$, and for $u\in\dom{\D}$, $\norm{(\Q-1)u}_{\Ls_\pm}\le\norm{\D u}_{\Ls_\pm}/\lambda$.
\item As $\lambda\to0$, $\Q$ converges strongly to the projection operator $\Q[0]$ defined in \autoref{def:projection operators 1.5d}.
\item For any $\lambda\ge0$, $\Q$ is symmetric.
\end{enumerate}
\end{lemma}
\begin{proof}
$\norm{\Q}_{\Ls_\pm}\le1$ follows from $\norm{(\D+\lambda)^{-1}}_{\Ls_\pm}\le\frac1{|\lambda|}$ as $i\D$ is self-adjoint and the nearest point of the spectrum of $\D$ is $0$. That $\norm{\Q}_{\bounded{\Ls_\pm}}=1$ is proved by observing that $\Q1=1$. Part (b) follows from the analyticity of resolvents as functions of $\lambda$. For (c) we compute using functional calculus for $u\in\dom{\D}$:
\begin{align*}
\norm{\Q u-u}_{\Ls_\pm}&=\norm{\left(\frac{\lambda^2}{\lambda^2-\D^2}-1\right)u}_{\Ls_\pm}=\norm{\frac{\D^2}{\lambda^2-\D^2}u}_{\Ls_\pm}\\
&\le\norm{\frac{\D}{\lambda+\D}}_{\bounded{\Ls_\pm}}\norm{\frac{1}{\lambda-\D}}_{\bounded{\Ls_\pm}}\norm{\D u}_{\Ls_\pm}\\
&\le1\cdot\frac1{\lambda}\cdot\norm{\D u}_{\Ls_\pm}\to0\text{\quad as $\lambda\to\infty$}
\end{align*}
and deduce the strong convergence $\Q\tos1$ by the density of $\dom{\D}$ in $\Ls_\pm$.

For (d) we introduce the spectral measure of the selfadjoint operator $-i\D$, which we denote by $M_\pm$. The projection onto $\ker(\D)$ is then $\Q[0]=M_\pm(\{0\})=\int_\mathbb{R}\chi(\alpha)\,dM_\pm(\alpha)$ where $\chi(0)=1$ and $\chi(\alpha)=0$ when $\alpha\neq0$. Recall that $\lambda(\lambda+\D)^{-1}=\int_\mathbb{R}\frac{\lambda}{\lambda+i\alpha}\,dM_\pm(\alpha)$. We compute for $u\in\Ls_\pm$,
\begin{align*}
\norm{\lambda(\lambda+\D)^{-1} u-M_\pm(\{0\})u}^2_{\Ls_\pm}&=\norm{\int_\mathbb{R}\left(\frac\lambda{\lambda+i\alpha}-\chi(\alpha)\right)\,dM_\pm(\alpha)u}^2_{\Ls_\pm}\\
&=\int_\mathbb{R}\left|\frac\lambda{\lambda+i\alpha}-\chi(\alpha)\right|^2\,d\norm{M_\pm(\alpha)u}^2_{\Ls_\pm},
\end{align*}
the last equality being due to orthogonality of spectral projections. This now tends to $0$ as $\lambda\to0$ by the dominated convergence theorem. Replacing $\D$ with $-\D$, which has the same kernel, we deduce that $\lambda(\lambda-\D)^{-1}\tos\Q[0]$. Finally we have $\Q=\lambda(\lambda-\D)^{-1}\lambda(\lambda+\D)^{-1}\tos(\Q[0])^2=\Q[0]$ by the composition of strong operator convergence.
To show (e) for $\lambda>0$ we simply note that $\D^2$ are selfadjoint, and extend to $\lambda=0$ by the strong operator convergence.
\end{proof}

These results carry through to the other operators.
\begin{lemma}\label{properties of J}
The operators $\J$ and $\oper{B}^\lambda$ have the properties:
\begin{enumerate}[(a)]
\item For all $\lambda\in[0,\infty)$, $\oper{B}^\lambda$ maps $L^2(\mathbb{R})$ into $L^2_0(\Omega)$ and $\J$ maps $L^2(\mathbb{R})\times L^2_0(\Omega)\to L^2(\mathbb{R})\times L^2_0(\Omega)$.
\item The familes $\{\J\}_{\lambda\in[0,\infty)}$ and $\{\oper{B}^\lambda\}_{\lambda\in[0,\infty)}$ are both uniformly bounded in operator norm.
\item Both $(0,\infty)\ni\lambda\mapsto\J$ and $(0,\infty)\ni\lambda\mapsto\oper{B}^\lambda$ are continuous in  the operator norm topology.
\item As $\lambda\to0$, $\J\to {\J[0]}$ and $\oper{B}^\lambda\to\oper{B}^0$ in the strong operator topology.
\item For any $\lambda\ge0$ the operator $\J$ is symmetric.
\item Let $\+{\oper{P}}$ be the multiplication operator acting in $L^2(\mathbb{R})\times L^2_0(\Omega)$ defined by
	\begin{equation*} 
	\+{\oper{P}}=\begin{bmatrix}\mathbbm{1}_\Omega&0\\0&\mathbbm{1}_\Omega\end{bmatrix}
	\end{equation*}
where $\mathbbm{1}_\Omega$ is the indicator function of the set $\Omega$. Then $\J=\J\+{\oper{P}}$.
\end{enumerate}
\end{lemma}
\begin{proof}
Part (a) is easily verifiable. We note that due to the relation 
\begin{equation*}
\oper{B}^\lambda=-\begin{bmatrix}0&1\end{bmatrix}\J\begin{bmatrix}
1\\0
\end{bmatrix}
\end{equation*}
it is sufficient to prove the results for $\J$.
We observe that due to the decay assumptions \eqref{eq:int-condition} on $\mu^\pm$, the moment 
\begin{equation*}
-\sum_{\pm}\int\mu^\pm\frac{1+v_1^2}{\<v^3}\,d\+v
\end{equation*}
is bounded in $L^\infty(\mathbb{R})$ and is real valued, so it is a bounded symmetric multiplication operator from $L^2(\mathbb{R})$ to $L^2(\mathbb{R})$. Next we decompose the second part of $\J$ as
\begin{equation}\label{J formula for properties}
\sum_{\pm}\int\mu_e^\pm\oper{T}_\pm(\Q-1)\oper{T}_\pm^*\begin{bmatrix}
\psi\\\phi
\end{bmatrix}
\,d\+v
\end{equation}
where $\oper{T}_\pm:\Ls_\pm\times\Ls_\pm\to\Ls_\pm$ is multiplication by the vector $\begin{bmatrix}
\hat{v}_2&-1
\end{bmatrix}$, and we have used the natural (and bounded) inclusions from $L^2(\mathbb{R})$ and $L^2_0(\Omega)$ into $\Ls_\pm$. Clearly $\oper{T}_\pm$ are bounded and we know that $\Q$ have bound 1 by \autoref{properties of Q}. Finally, we note that due to the decay assumptions on $\mu^\pm_e$ and its compact support in $x$, that multiplication by $\mu^\pm_e$ followed by integration $d\+v$ is bounded from $\Ls_\pm$ to $L^2(\mathbb{R})$ and $L^2(\Omega)$. Therefore $\J$ has a uniform bound in operator norm. Parts (c) and (d) then follow from the corresponding results for $\Q$ in \autoref{properties of Q} using \eqref{J formula for properties}.  (e) is clear from the symmetry of $\Q$ and \eqref{J formula for properties}. Finally (f) follows from the compact spatial support of $\mu^\pm,\mu^\pm_e,\mu^\pm_p$ inside $\Omega$.
\end{proof}

\begin{lemma}[Properties of $\oper{A}_1^\lambda$ and $\oper{A}_2^\lambda$]\label{lem:prop A}
Let $0\leq\lambda  <\infty$.
	\begin{enumerate}[(a)]
	\item The operator $\oper{A}_1^\lambda $ is self-adjoint on $L^2_0(\Omega)$ and the operator $\oper{A}_2^\lambda $ is self-adjoint on $L^2(\mathbb{R})$ with the respective domains $H^2_{0,n}(\Omega)$ and $H^2(\mathbb{R})$. 
	\item Both $[0,\infty)\ni\lambda\mapsto\oper{A}_1^\lambda$ and $[0,\infty)\ni\lambda\mapsto\oper{A}_2^\lambda$  are continuous  in the norm resolvent topology.
	 
	\item The spectrum of $\oper{A}_1^\lambda$ is purely discrete. The spectrum of $\oper{A}_2^\lambda$ is discrete and made up of finitely many eigenvalues in $(-\infty,\lambda^2)$ and continuous (possibly with embedded eigenvalues) in $[\lambda^2,\infty)$.
	 
	\item There exist constants $\gamma>0$ and $\Lambda  >0$ such that for all $\lambda  \geq\Lambda  $, $\oper{A}_i^\lambda>\gamma$, $i=1,2$.
	
	\end{enumerate}
\end{lemma}

\begin{proof}
Clearly $-\partial_x^2$ is symmetric. The perturbative terms are symmetric as well since $\Q$ are symmetric, see \autoref{properties of Q}. Self-adjointness is guaranteed by standard arguments, such as the Kato-Rellich theorem.

Let us prove (b), considering first $\oper{A}^\lambda_2$. It is sufficient to prove that $(\oper{A}^{\lambda}_2-i)^{-1}\to(\oper{A}^\sigma_2-i)^{-1}$ in operator norm as $\lambda\to\sigma$ (with $\lambda,\sigma\ge0$). 
We use the second resolvent identity to obtain
\begin{equation*}
\begin{aligned}
(\oper{A}^{\lambda}_2-i)^{-1}-(\oper{A}^\sigma_2-i)^{-1}&=(\oper{A}^{\lambda}_2-i)^{-1}(\oper{A}^{\sigma}_2-\oper{A}^{\lambda}_2)(\oper{A}^\sigma_2-i)^{-1}\\
&=(\oper{A}^{\lambda}_2-i)^{-1}((\sigma^2-\lambda^2)-(\oper{J}_{11}^\sigma-\oper{J}_{11}^\lambda))(\oper{A}^\sigma_2-i)^{-1}
\end{aligned}
\end{equation*}
where $\oper{J}_{11}^\lambda$ is the upper left component of $\+{\oper{J}}^\lambda$ written in block matrix form. Hence, as the resolvents are each bounded in operator norm by $1$,
\begin{equation*}
\norm{(\oper{A}^{\lambda}_2-i)^{-1}\!-(\oper{A}^\sigma_2-i)^{-1}}_{\bounded{L^2(\mathbb{R}}}\!\!\le |\sigma^2-\lambda^2|+\norm{(\oper{J}_{11}^\sigma-\oper{J}_{11}^\lambda)(\oper{A}^\sigma_2-i)^{-1}}_{\bounded{L^2(\mathbb{R}}}.
\end{equation*}
It thus suffices using \autoref{properties of J}(f) to show that $(\oper{J}_{11}^\sigma-\oper{J}_{11}^\lambda)\oper{P}(\oper{A}^\sigma_2-i)^{-1}\to0$ in operator norm, where $\oper{P}$ is the multiplication operator on $L^2(\mathbb{R})$ given by the indicator function of the set $\Omega$. $\oper{P}$ is relatively compact with respect to $-\partial_x^2$ by the Rellich theorem, and hence also relatively compact with respect to $\oper{A}^\sigma_2$ as it also has the domain $H^2(\mathbb{R})$ by part (a). Hence $\oper{P}(\oper{A}^\sigma_2-i)^{-1}$ is compact, which allows us to upgrade the strong convergence $\oper{J}_{11}^\lambda\tos \oper{J}_{11}^\sigma$ given by \autoref{properties of J} to operator norm convergence. The norm resolvent continuity of $\oper{A}^\lambda_2$ follows. The proof for $\oper{A}^\lambda_1$ is analogous, but lacking the $|\sigma^2-\lambda^2|$ term.

Part (c) is simple: both operators have a differential part (Laplacian) and a relatively compact perturbation. Hence both conclusions follow from Weyl's theorem  \cite[IV, Theorem 5.35]{Kato1995}. The finiteness of the discrete spectrum below the essential part in the case of $\oper{A}_2^\lambda$ is a consequence of \autoref{lemma:weyl-finite}. For part (d), we begin with $\oper{A}^\lambda_1$. Fix an arbitrary $h\in H^2_{0,n}(\Omega)$, then we compute,
\begin{equation*}
\ip{\oper{A}_1^\lambda h}{h}_{L^2_0(\Omega)}=\norm{\partial_x h}_{L^2_0(\Omega)}^2-\sum_\pm\iint\overline{h}\mu_e^\pm(1-\Q)h\,d\+vdx
\end{equation*}
Now we note as $h\in H^2_{0,n}(\Omega)$, $h$ is in $\dom{\D}$ when interpreted in $\Ls_\pm$. We now use \autoref{properties of Q}(c) to estimate,
\begin{equation*}
\begin{aligned}
\ip{\oper{A}_1^\lambda h}{h}_{L^2_0(\Omega)}&\ge \norm{\partial_x h}^2_{L^2_0(\Omega)}-\frac{1}{\lambda}\sum_\pm\norm{\mu_e^\pm/w^\pm}_{L^\infty(\Omega\times\mathbb{R}^3)}\norm{\D h}_{\Ls_\pm}\norm{h}_{\Ls_\pm}\\
&\ge \norm{\partial_x h}^2_{L^2_0(\Omega)}-\frac C{\sqrt{K}\lambda}\norm{\partial_x h}^2_{L^2_0(\Omega)}\\
&\ge K\norm{h}^2_{L^2_0(\Omega)}\left(1-\frac{C}{\sqrt{K}\lambda}\right)
\end{aligned}
\end{equation*}
where $K$ is the spectral gap of the Laplacian on the bounded domain $\Omega$, and we have used 
\begin{equation*}
\norm{\D h}^2_{\Ls_\pm}=\iint w^\pm|\hat{v}_1\partial_xh|^2\,d\+vdx\le\norm{\partial_xh}_{L^2_0(\Omega)}^2\sup_{x\in\Omega}\int w^\pm|\hat{v}_1|^2\,d\+v
\end{equation*}
and the natural bounded inclusions from $L^2_0(\Omega)$ into $\Ls_\pm$. We now just take $\Lambda>C/\sqrt{K}$.

For $\oper{A}^\lambda_2$ the proof is easier due to the $\lambda^2$ term. For $h\in H^2(\mathbb{R})$ we compute, using the formulation \eqref{eq:J lambda},
\begin{equation*}
\begin{aligned}
\ip{\oper{A}_2^\lambda h}{h}_{L^2(\mathbb{R})}&=\norm{\partial_xh}_{L^2(\mathbb{R})}^2+\lambda^2\norm{h}^2_{L^2(\mathbb{R})}-\ip{\J\begin{bmatrix}
h\\0
\end{bmatrix}}{\begin{bmatrix}
h\\0
\end{bmatrix}}_{L^2(\mathbb{R})\times L^2_0(\Omega)}\\
&\ge (\lambda^2-C')\norm{h}^2_{L^2(\mathbb{R})}
\end{aligned}
\end{equation*}
where we have used that the uniform bound in operator norm of $\J$ given by \autoref{properties of J}. We now take $\Lambda>\sqrt{C'}$.
\end{proof}

\begin{lemma}[Properties of $\M$]\label{lem:prop M}
For each $\lambda\in[0,\infty)$, the operator $\M$ is self-adjoint on $L^2(\mathbb{R})\times L^2_0(\Omega)$ with domain $H^2(\mathbb{R})\times H^2_{0,n}(\Omega)$. For any $\lambda\ge0$, the operator $\M$ has essential spectrum $[\lambda^2,\infty)$. The family $\{\M\}_{\lambda\in[0,\infty)}$ is continuous in the norm resolvent topology.
\end{lemma}

\begin{proof}
The proof essentially mimics (and uses) the preceding proofs, and is therefore left for the reader.
\end{proof}

\subsection{The cylindrically symmetric case}\label{sec:prop op cyl}\label{sec:prop operators 3d}
As many of the proofs are the same as in the 1.5 dimensional case above, we give the details only where they differ.
\begin{lemma}\label{properties of Qt}
In the respective spaces $\Ns_\pm$, $\Qtsym$ and $\Qtskew$ satisfy:
\begin{enumerate}[(a)]
 \item $\norm{\Qtsym}_{\bounded{\Ns_\pm}}=1$ and $\norm{\Qtskew}_{\bounded{\Ns_\pm}}\le\frac12$. 

\item $\Qtsym$ and $\Qtskew$ can be extended from $\lambda>0$ to $\operatorname{Re}\lambda>0$ as holomorphic operator valued functions. In particular they are continuous for $\lambda>0$ in operator norm topology.
\item As $\mathbb{R}\ni\lambda\to\infty$, $\Qtsym\tos1$, and for $u\in\dom{\Dt}$, we have the bound $\norm{(\Qtsym-1)u}_{\Ns_\pm}\le\norm{\Dt u}_{\Ns_\pm}/\lambda$.
\item As $\mathbb{R}\ni\lambda\to\infty$, $\Qtskew\tos0$, and for $u\in\dom{\Dt}$, we have the bound $\norm{\Qtsym{}u}_{\Ns_\pm}\le\norm{\Dt u}_{\Ns_\pm}/\lambda$.
\item As $\lambda\to0$, $\Qtsym$ converge strongly to the projection operators $\Qt[0]$ defined in \autoref{def:projection operators 3d}.
\item As $\lambda\to0$, $\Qtskew$ converge strongly to $0$.
\item For any $\lambda\ge0$, $\Qtsym$ are symmetric and $\Qtskew$ are skew-symmetric.
\end{enumerate}
\end{lemma}
\begin{proof}
The claims about $\Qtsym$ may be proved in the same way as those in \autoref{properties of Q}. For (a), the spectral theorem applied to the self-adjoint operators $-i\Dt$ implies that $\Qtskew$ are unitarily equivalent to a multiplication operator, so that
\begin{equation*}
\norm{\Qtskew}_{\bounded{\Ns_\pm}}=\norm{\frac{-i\alpha\lambda}{\lambda^2+\alpha^2}}_{L^\infty_\alpha(\operatorname{sp}(i\Dt))}\le \norm{\frac{-i\alpha\lambda}{\lambda^2+\alpha^2}}_{L^\infty_\alpha(\mathbb{R})}=\frac12.
\end{equation*}
The proof of (b) follows, as in the proof of \autoref{properties of Q}, from the holomorphicity of the resolvent. For (d), we let $u\in\dom{\Dt}$, and then for $\lambda>0$ we have,
\begin{align*}
\norm{\Qtskew{}u}_{\Ns_\pm}&\le \lambda \norm{(\lambda^2+\Dt^2)^{-1}}_{\bounded{\Ns_\pm}} \norm{\Dt u}_{\Ns_\pm}\\
&\le \frac{1}{\lambda}\norm{\Dt u}_{\Ns_\pm}\to 0\text{   as } \lambda\to\infty.
\end{align*}
The strong convergence to $0$ then follows from the density of $\dom{\Dt}$. For (f), we repeat the proof of \autoref{properties of Q}, noting that it is shown that $\lambda(\lambda+\Dt)^{-1}\tos\Qt[0]$ and $\Qtsym{}\tos\Qt[0]$ as $\lambda\to0$. That $\Qtskew{}\tos0$ as $\lambda\to0$ now follows from the identity, valid for all $\lambda>0$,
\begin{equation*}
\lambda(\lambda+\Dt)^{-1}=\Qtsym{}+\Qtskew{}.
\end{equation*}
Finally, (g) is a consequence of \autoref{properties of D 2}.
\end{proof}

\begin{lemma}\label{properties of J 3d}
The operators $\Jt$, and $\widetilde{\oper{B}}_i^\lambda$ for $i=1,2,3,4$ have the properties:
\begin{enumerate}[(a)]
\item For all $\lambda\in[0,\infty)$, $\widetilde{\oper{B}}_1^\lambda\in \bounded{L^2_\theta(\mathbb{R}^3;\mathbb{R}^3),L^2_{cyl}(\mathbb{R}^3)}$, $\widetilde{\oper{B}}_2^\lambda\in \bounded{L^2_\theta(\mathbb{R}^3;\mathbb{R}^3),L^2_{rz}(\mathbb{R}^3;\mathbb{R}^3)}$, $\widetilde{\oper{B}}_3^\lambda\in \bounded{L^2_{cyl}(\mathbb{R}^3),L^2_{rz}(\mathbb{R}^3;\mathbb{R}^3)}$ and $\widetilde{\+{\oper{B}}}_4^\lambda\in\bounded{L^2_{\theta}(\mathbb{R}^3;\mathbb{R}^3),L^2_{cyl}(\mathbb{R}^3)\times L^2_{rz}(\mathbb{R}^3;\mathbb{R}^3)}$ with bounds uniform in $\lambda$.
\item Each of $(0,\infty)\ni\lambda\mapsto\Jt$ and $(0,\infty)\ni\lambda\mapsto\widetilde{\oper{B}}_i^\lambda$, $i=1,2,3,4$ are continuous in the operator norm topology.
\item As $\lambda\to0$, $\Jt\to {\Jt[0]}$, $\widetilde{\oper{B}}_1^\lambda\to\widetilde{\oper{B}}_1^0$, $\widetilde{\oper{B}}_2^\lambda\to0$ and $\widetilde{\oper{B}}_3^\lambda\to0$ in the strong topology.
\item For any $\lambda\ge0$ the operator $\Jt$ is symmetric.
\item Let $\widetilde{\+{\oper{P}}}$ be the multiplication operator acting in $L^2_{cyl}(\mathbb{R}^3)\times L^2_\theta(\mathbb{R}^3;\mathbb{R}^3)\times L^2_{rz}(\mathbb{R}^3;\mathbb{R}^3)$ defined by
	\begin{equation*} 
	\widetilde{\+{\oper{P}}}=\begin{bmatrix}\mathbbm{1}_\Omega&0&0\\0&\mathbbm{1}_\Omega&0\\0&0&\mathbbm{1}_\Omega\end{bmatrix}
	\end{equation*}
where $\mathbbm{1}_\Omega$ is the indicator function of the set $\Omega$. Then $\Jt=\Jt\widetilde{\+{\oper{P}}}$.
\end{enumerate}
\end{lemma}
\begin{proof}
That the operators map the corresponding spaces to each other may be verified directly from \eqref{eq:operators3d}, noting in particular the notation $\+{\hat{v}}_{\theta}=\+e_\theta\hat{v}_\theta$ and $\+{\hat{v}}_{rz}=\+e_r\hat{v}_r+\+e_z\hat{v}_z$. As in the proof of \autoref{properties of J}, the uniform (in $\lambda$) bound on the operator norms may be obtained using the decay assumptions on the equilibrium and the uniform bound on the norms of $\Qtsym$ and $\Qtskew$ given by \autoref{properties of Qt}. In the same way (c) and (d) follow from the corresponding results in \autoref{properties of Qt}. 

To show the symmetry of $\Jt$ for $\lambda>0$ we use the block matrix form, noting that $\Qtsym$ appears on the diagonal, and that the off diagonal entries have $\widetilde{\oper{B}}_i^\lambda$ and their adjoints in the appropriate configuration. Then we extend to $\lambda=0$ by the strong convergence. As in \autoref{properties of J} (e) follows from the spatial support properties of the equilibrium.
\end{proof}
\begin{lemma}[Properties of $\widetilde{\oper{A}}_1^\lambda$, $\widetilde{\oper{A}}_2^\lambda$, $\widetilde{\oper{A}}_3^\lambda$ and $\widetilde{\+{\oper{A}}}_4^\lambda$]\label{lem:prop A 3d}
Let $0\leq\lambda  <\infty$.
	\begin{enumerate}[(a)]
	\item The operator $\widetilde{\oper{A}}_1^\lambda$ is self-adjoint on $L^2_{cyl}(\mathbb{R}^3)$, the operator $\widetilde{\oper{A}}_2^\lambda $ is self-adjoint on $L^2_\theta(\mathbb{R}^3;\mathbb{R}^3)$, $\widetilde{\oper{A}}_3^\lambda$ is self-adjoint on $L^2_{rz}(\mathbb{R}^3;\mathbb{R}^3)$ and $\widetilde{\+{\oper{A}}}_4^\lambda$ is self-adjoint on $L^2_{cyl}(\mathbb{R}^3)\times L^2_{rz}(\mathbb{R}^3;\mathbb{R}^3)$ with the respective domains $H^2_{cyl}(\mathbb{R}^3)$, $H^2_{\theta}(\mathbb{R}^3;\mathbb{R}^3)$, $H^2_{rz}(\mathbb{R}^3;\mathbb{R}^3)$ and $H^2_{cyl}(\mathbb{R}^3)\times H^2_{rz}(\mathbb{R}^3;\mathbb{R}^3)$.
	\item The mappings $[0,\infty)\ni\lambda\mapsto\widetilde{\oper{A}}_1^\lambda$, $[0,\infty)\ni\lambda\mapsto\widetilde{\oper{A}}_2^\lambda$, $[0,\infty)\ni\lambda\mapsto\widetilde{\oper{A}}_3^\lambda$ and
	$[0,\infty)\ni\lambda\mapsto\widetilde{\+{\oper{A}}}_4^\lambda$ are continuous  in the norm resolvent topology.
	 
	\item The spectra of $\widetilde{\oper{A}}_1^\lambda$, $\widetilde{\oper{A}}_2^\lambda$, $\widetilde{\oper{A}}_1^\lambda$ and $\widetilde{\+{\oper{A}}}_4^\lambda$ are discrete and finite in $(-\infty,\lambda^2)$  and continuous (possibly with embedded eigenvalues) in $[\lambda^2,\infty)$.
	 
	\item There exist constants $\gamma>0$ and $\Lambda>0$ such that for all $\lambda  \geq\Lambda  $, $\widetilde{\oper{A}}_i^\lambda>\gamma$, $i=1,2,3,4$.
	\end{enumerate}
\end{lemma}
\begin{proof}
The proof for each of $\widetilde{\oper{A}}_i^\lambda$, $i=1,2,3,4$ is analogous to that of \autoref{lem:prop A} for $\oper{A}_2^\lambda$. We omit the details.
\end{proof}

\begin{lemma}[Properties of $\Mt$]\label{lem:prop mt}
For each $\lambda\in[0,\infty)$, the operator $\widetilde{\+{\oper{M}}}^\lambda$ is self-adjoint on $L^2_\theta(\mathbb{R}^3;\mathbb{R}^3)\times L^2_{cyl}(\mathbb{R}^3)\times L^2_{rz}(\mathbb{R}^3;\mathbb{R}^3)$ with domain $H^2_\theta(\mathbb{R}^3;\mathbb{R}^3)\times H^2_{cyl}(\mathbb{R}^3)\times H^2_{rz}(\mathbb{R}^3;\mathbb{R}^3)$. For any $\lambda\ge0$, the operator $\widetilde{\+{\oper{M}}}^\lambda$ has essential spectrum $(-\infty,-\lambda^2]\cup[\lambda^2,\infty)$. The family $\{\widetilde{\+{\oper{M}}}^\lambda\}_{\lambda\in[0,\infty)}$ is continuous in the norm resolvent topology.
\end{lemma}
\begin{proof}
This is again analogous to the previous proofs, and is left to the reader.
\end{proof}

\section{Construction of example equilibria}\label{sec:examples}
In this section we will provide (non-optimal) existence criteria for compactly supported equilibria of the $1.5d$ system which can be written in the form \eqref{eq:jeans theorem ansatz}-\eqref{eq:constants-1.5}. Examples of equilibria in the cylindrically symmetric case were already provided in \cite{Lin2007}.
\begin{proposition}[Existence of confined equilibria]\label{lem:existence-of-confined-equilibria}
Let $R>0,\alpha>2$ and $A^\pm\subset \mathbb{R}^2$ be bounded domains. Then there are constants $c,C>0$ such that if two functions $\mu^\pm(e^\pm,p^\pm)\in C^1_0(\mathbb{R}^2)$ with support  in $A^\pm$ satisfy
	\begin{equation*}
	|\mu^\pm|,|\mu_e^\pm|,|\mu_p^\pm|\le c(1+|e^\pm|)^{-\alpha}
	\end{equation*}
and a function $\psi^{ext}\in H^2_{loc}(\mathbb{R})$ satisfies
	\begin{equation*}
	|\psi^{ext}(x)|\ge C(1+|x|^2)\text{ for }|x|\ge R
	\end{equation*} 
then there are potentials $\phi^0,\psi^0\in H^2_{loc}(\mathbb{R})$ such that $(\mu^\pm(e^\pm,p^\pm),\phi^0,\psi^0,\psi^{ext},\phi^{ext}=0)$ is an equilibrium of the $1.5d$ relativistic Vlasov-Maxwell equations \eqref{eq:rvm-1.5} with spatial support in $[-R,R]$, where the relationship between $(x,v_1,v_2)$ and $(e^\pm,p^\pm)$ is as defined in \eqref{eq:constants-1.5}.
\end{proposition}
\begin{remark}[Trivial solutions]
Of course, the result does not say that the obtained equilibrium is not everywhere zero. This may be ruled out by choosing $\mu^\pm$ and $\psi^{ext}$ in such a way that (for example) $\mu^\pm(x=0,v=0)>0$ if $\phi^0,\psi^0\equiv0$. Let us sketch the argument. Recall that we write $f^{0,\pm}(x,v)=\mu^\pm(\left<\+v\right>\pm\phi^0(x),v_2\pm\psi^0(x)\pm\psi^{ext}(x))=\mu^\pm(e^\pm,p^\pm)$. If $f^{0,\pm}(x,\+v)=0$ for all $(x,\+v)$ then $\rho,j_i=0$ and $\phi^0,\psi^0=0$ for all $x$. Therefore $e^\pm=\<{\+v}$ and $p^\pm=v_2\pm\psi^{ext}(x)$, and
\begin{equation*}
f^{0,\pm}(0,0)=\mu^\pm(1,\pm\psi^{ext}(0)).
\end{equation*}
The RHS is something we can ensure is positive by choosing $A^\pm$, $\mu^\pm$ and $\psi^{ext}$ appropriately. Under this appropriate choice one obtains a contradiction.
\end{remark}

\begin{proof}[Proof of \autoref{lem:existence-of-confined-equilibria}]
Given two elements $\rho,j_2\in L^2(\mathbb{R})$ with compact support, we define
	\begin{equation*}
	\phi^0=G*\rho,\qquad\psi^0=G*j_2,\quad
	\end{equation*}
where $G(x)=-|x|/2$ is the fundamental solution of the Laplacian in one dimension. [We note that one expects $j_1$ to vanish for an equilibrium,  due to parity in $v_1$, hence it does not appear in the setup.] Thus we define $e^\pm=e^\pm(x,v_1,v_2)$ and $p^\pm=p^\pm(x,v_1,v_2)$ via the usual relations  \eqref{eq:constants-1.5}, which we recall  for the reader's convenience:
	\begin{equation*}
	e^\pm(x,\+v)
	=
	\left<\+v\right>\pm\phi^0(x),\qquad
	p^\pm(x,\+v)
	=
	v_2\pm\psi^0(x)\pm\psi^{ext}(x)
	\end{equation*}
($\phi^{ext}$ is zero).
We let $\mathcal{F}:L^2(\mathbb{R})^2\to L^2(\mathbb{R})^2$ be the (non-linear) map defined by
\begin{equation*}\label{eq:fixed-point-operator-F}
\mathcal{F}\begin{bmatrix}\rho\\ j_2\end{bmatrix}=\int \begin{bmatrix}1\\\hat{v}_2\end{bmatrix}(\mu^+(e^+,p^+)-\mu^-(e^-,p^-))\,d\+v.
\end{equation*}
 A fixed point of $\mathcal{F}$ is the charge and current densities of an equilibrium $(\mu^\pm(e^\pm,p^\pm),\phi^0,\psi^0,\psi^{ext},\phi^{ext}=0)$. We define $X\subseteq L^2(\mathbb{R})^2$ as
\begin{equation*}
X=\{(\rho,j_2)\in L^2(\mathbb{R})^2\;:\;\text{ both supported in }[-R,R]\text{ and bounded by }C'\}
\end{equation*}
for a positive constant $C'$ to be chosen. This set is clearly convex. We will show that for $c>0$ sufficiently small and $C>0$ sufficiently large, $\mathcal{F}$ is a compact continuous map  $X\hookrightarrow X$  and thus, by the Schauder fixed point theorem, has a fixed point. 

\paragraph{Step 1: Compact support.} $C'$ and $C$ can be chosen so that $\mathcal{F}$ maps $X$ into functions supported in $[-R,R]$.

For $(\rho,j_2)\in X$ and $|x|>R$, we have,
\begin{equation*} 
|\phi^0(x)|=|(G*\rho)(x)|\le C'\int_{-R}^R|G(x-y)|\,dy= \frac{C'}2\int^R_{-R}|x-y|\,dy= C'R|x|
\end{equation*}
and the same bound holds for $\psi^0$. This allows us to control $v_2$ using $e^\pm$ and $x$. Indeed,
\begin{equation*}
|v_2|\le \<{\+v} = e^\pm\mp\phi^0(x)\le |e^\pm|+|\phi^0(x)|\le |e^\pm|+C'R|x|.
\end{equation*}
Which gives the following lower bound for $|p^\pm|+|e^\pm|$ in terms of $x$:
\begin{equation*}
\begin{aligned}
|p^\pm|+|e^\pm|&=|v_2\pm\psi^0(x)\pm\psi^{ext}(x)|+|e^\pm|\ge \psi^{ext}(x)-|v_2|-|\psi^0(x)|+|e^\pm|\\
&\ge \psi^{ext}(x)-|e^\pm|-2C'R|x|+|e^\pm|\\
&\ge  C(1+|x|^2)-2C'R|x|.
\end{aligned}
\end{equation*}
By taking $C'$ small enough and $C$ large enough we can ensure that if $|x|>R$ then $(e^\pm,p^\pm)$ lie outside any disc in $\mathbb{R}^2$, and in particular outside $A^\pm$, where $\mu^\pm$ are supported. This proves the claim.

\paragraph{Step 2: Uniform $L^\infty$ bound.} $C'$ and $c$ can be chosen so that $\mathcal{F}$ maps to functions with $L^\infty$ norm smaller than $C'$.

Estimating $\phi^0(x)$ for $|x|\leq R$
\begin{equation*} 
|\phi^0(x)|\leq\frac{C'}2\int^R_{-R}|x-y|\,dy= \frac{C'}{2}(x^2+R^2),
\end{equation*}
we take $C'$ small enough (recall it was already taken to be small in the previous step, hence we may require it to be even smaller) so that $|\phi^0(x)|\le 3/4$ for $|x|\le R$. Now the decay assumption on $\mu^\pm$ allows us to show a uniform bound on $|\mathcal{F}_1(\rho,j_2)(x)|$ in $|x|\le R$:
\begin{equation}\label{eq:uniform-bound-on-F_1}
\begin{aligned}
|\mathcal{F}_1(\rho,j_2)(x)|&\le \sum_{\pm}\int |\mu^\pm(e^\pm,p^\pm)|\,d\+v
\le \sum_{\pm}\int\frac{c}{(1+|e^\pm|)^\alpha}\,d\+v\\
&\le \sum_{\pm}\int \frac{c}{(1+\<{\+v}-|\phi^0(x)|)^\alpha}\,d\+v
\le \sum_{\pm}\int \frac{c}{(1+\<{\+v}-\frac34)^\alpha}\,d\+v\\
&= \int \frac{2c}{(\frac14+\<{\+v})^\alpha}\,d\+v=C''c<\infty.
\end{aligned}
\end{equation}
We can bound $|\mathcal{F}_2(\rho,j_2)(x)|$ in the same way as $|\hat{v}|\le1$. Finally we choose $c$ so that $C''c\le C'$.

\paragraph{Step 3: Uniform $L^\infty$ bound on the derivative.} There is a constant $C'''$ such that for any $(\rho,j_2)\in X$, we have $\norm{\partial_x\mathcal{F}_1(\rho,j_2)}_{L^\infty[-R,R]}\le C'''$ and $\norm{\partial_x\mathcal{F}_2(\rho,j_2)}_{L^\infty[-R,R]}\le C'''$.

We compute for $\mathcal{F}_1$, and note that $\mathcal{F}_2$ is analogous. Using the chain rule, we have
\begin{equation*}
\begin{aligned}
\partial_x&\int\left(\mu^+(e^+,p^+)-\mu^-(e^-,p^-)\right)\,dv=\\
&(\partial_x\phi^0)\int(\mu^+_e(e^+,p^+)+\mu^-_e(e^-,p^-))\,dv\\
&\quad+(\partial_x\psi^0+\partial_x\psi^{ext})\int(\mu^+_p(e^+,p^+)+\mu^-_p(e^-,p^-))\,dv.
\end{aligned}
\end{equation*}
The two integrals are bounded uniformly in $x$ by the arguments in Step 2 using the corresponding assumed bounds on $\mu_e^\pm$ and $\mu_p^\pm$ respectively. As the external field $\phi^{ext}$ lies in $H^2_{loc}(\mathbb{R})$, its derivative $\partial_x\phi^{ext}$ lies in $H^1([-R,R])$ and is bounded in $L^\infty([-R,R])$ by Morrey's inequality. It remains to bound $\partial_x\phi^0$ and $\partial_x\psi^0$ uniformly for all $x\in[-R,R]$. These are controlled directly using the Green's function $G(x)$ and uniform bounds of Step 2. Indeed,
\begin{equation*}
|(\partial_x\phi^0)(x)|=|((\partial_xG)*\rho)(x)|\le\frac{C'}2\int^R_{-R}|\operatorname{sign}(x-y)|\,dy\le C'R
\end{equation*}
and the computation for $\partial_x\psi^0$ is identical.
\paragraph{Step 4: $\mathcal{F}$ is a compact continuous map from $X$ to $X$.}

Steps 1 and 2 imply that $\mathcal{F}(X)\subseteq X$. Step 3 and the Rellich theorem imply that $\mathcal{F}(X)$ is relatively compact in $X$. It remains to show that $\mathcal{F}$ is continuous. This may be shown using dominated convergence and the bounds in Step 2. Indeed, suppose that $(\rho^n,j_2^n)\subseteq X$ is a sequence converging to $(\rho,j_2)\in X$ strongly in $L^2(\mathbb{R})^2$. We shall show that $\mathcal{F}_1(\rho^n,j_2^n)\to\mathcal{F}_1(\rho,j_2)$ in $L^2$, the result for $\mathcal{F}_2$ is analogous. By Step 2 and dominated convergence it is enough to show convergence pointwise, i.e. for each $x\in[-R,R]$. Next, by \eqref{eq:uniform-bound-on-F_1} and dominated convergence again, it is sufficient to show that the corresponding densities $\mu^\pm(e^\pm,p^\pm)$ converge pointwise in $(x,\+v)$. Continuity of $\mu^\pm$ reduces this to showing pointwise convergence of the corresponding microscopic energy and momenta $e^\pm$ and $p^\pm$. The definitions of these quantities imply that it is enough to show that the corresponding electric and magnetic potentials $\phi^{0,n}$ and $\psi^{0,n}$ converge pointwise. Finally, as the potentials are $(\rho^n,j_2^n)$ convolved with $G(x)=-|x|/2$, the convergence $(\rho^n,j_2^n)\to (\rho,j_2)$ in $L^2([-R,R])^2$ gives the required pointwise convergence.

\bigskip
This concludes the proof.
\end{proof}

\bibliography{library}

\begin{thebibliography}{10}

\bibitem{Ben-Artzi2011}
Jonathan Ben-Artzi.
\newblock {Instability of nonmonotone magnetic equilibria of the relativistic
  Vlasov-Maxwell system}.
\newblock {\em Nonlinearity}, 24(12):3353--3389, December 2011.

\bibitem{Ben-Artzi2011b}
Jonathan Ben-Artzi.
\newblock {Instability of nonsymmetric nonmonotone equilibria of the
  Vlasov-Maxwell system}.
\newblock {\em J. Math. Phys.}, 52(12):123703, 2011.

\bibitem{Ben-Artzi2013e}
Jonathan Ben-Artzi and Thomas Holding.
\newblock {Approximations of strongly continuous families of unbounded
  operators}.
\newblock {\em Preprint arXiv1403.3963}, 17 pages, May 2015.

\bibitem{Glassey1988a}
Robert~T. Glassey and Jack~W. Schaeffer.
\newblock {Control of velocities generated in a two—dimensional collisionless
  plasma with symmetry}.
\newblock {\em Transp. Theory Stat. Phys.}, 17(5-6):467--560, 1988.

\bibitem{Glassey1990}
Robert~T. Glassey and Jack~W. Schaeffer.
\newblock {On the ‘one and one-half dimensional’ relativistic
  Vlasov-Maxwell system}.
\newblock {\em Math. Methods Appl. Sci.}, 13(2):169--179, August 1990.

\bibitem{Marsden1985}
Darryl~D. Holm, Jerrold~E. Marsden, Tudor Ratiu, and Alan Weinstein.
\newblock {Nonlinear stability of fluid and plasma equilibria}.
\newblock {\em Phys. Rep.}, 123(1\&2):1--116, 1985.

\bibitem{Jeans1915}
J.~H Jeans.
\newblock {On the theory of star-streaming and the structure of the universe}.
\newblock {\em Mon. Not. R. Astron. Soc.}, 76:70--84, 1915.

\bibitem{Kato1995}
Tosio Kato.
\newblock {\em {Perturbation Theory for Linear Operators}}.
\newblock Springer-Verlag, 1995.

\bibitem{Lin2007}
Zhiwu Lin and Walter~A. Strauss.
\newblock {Linear stability and instability of relativistic Vlasov-Maxwell
  systems}.
\newblock {\em Commun. Pure Appl. Math.}, 60(5):724--787, May 2007.

\bibitem{Lin2008}
Zhiwu Lin and Walter~A. Strauss.
\newblock {A sharp stability criterion for the Vlasov-Maxwell system}.
\newblock {\em Invent. Math.}, 173(3):497--546, April 2008.

\bibitem{Marchioro1986}
Carlo Marchioro and Mario Pulvirenti.
\newblock {A note on the nonlinear stability of a spatially symmetric
  vlasov-possion flow}.
\newblock {\em Math. Methods Appl. Sci.}, 8(1):284--288, June 1986.

\bibitem{Penrose1960}
Oliver Penrose.
\newblock {Electrostatic Instabilities of a Uniform Non-Maxwellian Plasma}.
\newblock {\em Phys. Fluids}, 3(2):258--265, 1960.

\bibitem{Tao2011b}
Terence Tao.
\newblock {The spectral theorem and its converses for unbounded symmetric
  operators (from the blog ``What's New?''
  http://terrytao.wordpress.com/2011/12/20/the-spectral-theorem-and-its-converses-for-unbounded-symmetric-operators/)},
  2011.

\bibitem{Wollman1984}
Stephen Wollman.
\newblock {An existence and uniqueness theorem for the Vlasov-Maxwell system}.
\newblock {\em Commun. Pure Appl. Math.}, 37(4):457--462, July 1984.

\end{thebibliography}
\bibliographystyle{plain}
\end{document}